%% file: accepted_19_06.tex
\documentclass[microtype]{gtpart}     
\gtart
\usepackage{stmaryrd} 
\usepackage{upgreek} 
\usepackage{faktor} 
\usepackage{tikz} 
\usetikzlibrary{cd} 
\usepackage{enumitem} 
\usepackage{mathtools} 
\usepackage[mathcal]{euscript} 
\usepackage[T1]{fontenc} 

\title{Edge stabilization in the homology of graph braid groups}
\author{Byung Hee An}
\givenname{Byung Hee}
\surname{An}
\address{Center for Geometry and Physics, Institute for Basic Science (IBS), Pohang 37673, Republic of Korea}
\email{anbyhee@ibs.re.kr}
\author{Gabriel C. Drummond-Cole}
\givenname{Gabriel C.}
\surname{Drummond-Cole}
\address{Center for Geometry and Physics, Institute for Basic Science (IBS), Pohang 37673, Republic of Korea}
\email{gabriel.c.drummond.cole@gmail.com}
\urladdr{drummondcole.com/gabriel/academic/}

\author{Ben Knudsen}
\givenname{Ben}
\surname{Knudsen}
\address{Department of Mathematics, Harvard University, Cambridge 02138, USA}
\email{knudsen@math.harvard.edu}

\arxivreference{1806.05585}

\newtheorem*{formality}{Formality Theorem}
\newtheorem*{degreetheorem}{Degree Theorem}

\theoremstyle{definition}
\newtheorem{definition}{Definition}[section]
\newtheorem{notation}[definition]{Notation}

\newtheorem{example}[definition]{Example}

\newtheorem{construction}[definition]{Construction}
\newtheorem{observation}[definition]{Observation}

\theoremstyle{plain}
\newtheorem{proposition}[definition]{Proposition}
\newtheorem{lemma}[definition]{Lemma}
\newtheorem{corollary}[definition]{Corollary}
\newtheorem{theorem}[definition]{Theorem}

\theoremstyle{remark}
\newtheorem{remark}[definition]{Remark}

\makeatletter
\@addtoreset{definition}{section}
\makeatother

\usetikzlibrary{shapes.geometric}
\usetikzlibrary{decorations.markings}
\usetikzlibrary{patterns,decorations.pathreplacing}

\newcommand{\gaps}{\mathcal{G}}

\newcommand{\localstates}[1]{S({#1})}
\newcommand{\reducedlocalstates}[1]{\widetilde{S}({#1})}
\newcommand{\intrinsic}[1]{S(#1)}

\newcommand{\fullyreduced}[1]{\widetilde{S}\left(#1\right)}

\newcommand{\starreduced}[1]{S^{\Yup}\left(#1\right)}

\newcommand{\sing}{\mathrm{sing}}

\newcommand{\Gph}{\mathcal{G}\mathrm{ph}}

\newcommand{\Conf}{\mathrm{Conf}}

\newcommand{\subd}{\mathcal{P}}

\DeclareMathOperator*{\hocolim}{\mathrm{hocolim}}

\DeclareMathOperator*{\Tor}{\mathrm{Tor}}
\DeclareMathOperator*{\colim}{\mathrm{colim}}

\newcommand{\longends}{\mathcal{L}}
\newcommand{\Mod}{\mathcal{M}\mathrm{od}}
\newcommand{\Wvring}{R_{W,v}}
\newcommand{\Wvcomplex}{C_{W,v}}
\newcommand{\fullyreducedsub}[2]{\widetilde{S}_{#1}(#2)}
\newcommand{\fullyreducedsubnoarg}[1]{\widetilde{S}_{#1}}
\newcommand{\growthdeg}[3]{D_{#1}^{#2}(#3)}
\newcommand{\Ramos}[2]{\Delta_{#1}^{#2}}
\newcommand{\essential}{V^{\mathrm{ess}}}

\definecolor{coloryellow}{RGB}{240,228,66}
\definecolor{colorskyblue}{RGB}{86,180,233}
\definecolor{colorvermillion}{RGB}{213,94,0}

\newcommand{\Agraph}{\graphfont{A}}

\newcommand{\graphfont}{\mathsf}

\newcommand{\handcuffgraph}{\graphfont{H}}
\newcommand{\thetagraph}[1]{\graphfont{\Theta}_{#1}}
\newcommand{\completegraph}[1]{\graphfont{K}_{#1}}
\newcommand{\intervalgraph}{\graphfont{I}}

\newcommand{\stargraph}[1]{\graphfont{S}_{#1}}
\newcommand{\graf}{\graphfont{\Gamma}}
\newcommand{\Xigraph}{\graphfont{\Xi}}

\newcommand{\lollipopgraph}[1]{\graphfont{L}_{#1}}
\newcommand{\figureeightgraph}{\graphfont{8}}
\newcommand{\cyclegraph}[1]{\graphfont{C}_{#1}}
\newcommand{\st}[1]{\mathrm{st}_{#1}}
\newcommand{\field}{\mathbb{F}}

\newsavebox{\foobox}
\newcommand{\slantbox}[3]
  {%
    \mbox
      {%
        \sbox{\foobox}{#3}%
        \hskip\wd\foobox
        \pdfsave
        \pdfsetmatrix{#1 #2 0 1}%
        \llap{\usebox{\foobox}}%
        \pdfrestore
      }%
  }

\begin{document}
\begin{abstract}
We introduce a novel type of stabilization map on the configuration spaces of a graph, which increases the number of particles occupying an edge.
There is an induced action on homology by the polynomial ring generated by the set of edges, and we show that this homology module is finitely generated.
An analogue of classical homological and representation stability for manifolds, this result implies eventual polynomial growth of Betti numbers. We calculate the exact degree of this polynomial, in particular verifying an upper bound conjectured by Ramos.
Because the action arises from a family of continuous maps, it lifts to an action at the level of singular chains, which contains strictly more information than the homology level action. We show that the resulting differential graded module is almost never formal over the ring of edges.
\end{abstract}
\maketitle
\section{Introduction}

Configuration spaces of manifolds have numerous applications in algebraic topology and homotopy theory---see Arnold~\cite{Arnold:CRGDB}, McDuff~\cite{McDuff:CSPNP}, and Cohen--Lada--May~\cite{CohenLadaMay:HILS} for a few notable examples. 
More recently, there has been a growing swell of interest in configuration spaces of \emph{graphs}. 
For a graph $\graf$---which is to say a finite, $1$-dimensional cell complex---we write \[B_k(\graf)=\left\{(x_1,\ldots, x_k)\in \graf^k: x_i\neq x_j\text{ if }i\neq j\right\}_{/\Sigma_k}\] for the $k$th unordered \emph{configuration space} of $\graf$. 
First considered by Ghrist~\cite{Ghrist:CSBGGR} from the point of view of robotics and motion planning, these spaces were shown by Abrams~\cite[Corollary 3.11]{Abrams:CSBGG} to be aspherical for $\graf$ connected and so classify their respective fundamental groups, the graph braid groups of $\graf$. 
Much effort has been dedicated to understanding the homological, geometric, and combinatorial properties of these groups by a number of different researchers and groups. A non-exhaustive list includes~%
\cite{AnDrummond-ColeKnudsen:SSGBG,ChettihLuetgehetmann:HCSTL,FarleySabalka:DMTGBG,FarleySabalka:OCRTBG,HarrisonKeatingRobbinsSawicki:NPQSG,KimKoPark:GBGRAAG,KoPark:CGBG,Luetgehetmann:RSCSG,LutgehetmannRecioMitter:TCCSFAGBG,MaciazekSawicki:HGP1CG,MaciazekSawicki:NAQSG,Ramos:SPHTBG,Sabalka:ORIPTBG,Swiatkowski:EHDCSG}.

\subsection{Stability phenomena} 
Configuration spaces of different cardinalities relate to one another in a variety of ways, and it is often simpler to study the graded space $B(-)=\bigsqcup_{k\geq0}B_k(-)$. 
For example, if $M$ is a manifold with non-empty boundary, there is a \emph{stabilization map} \[\sigma:B(M)\to B(M)\] which inserts a new particle near the boundary. 
This map increases the cardinality $k$ by $1$, and, for sufficiently large $k$, a theorem of McDuff asserts that the induced map $H_i(B_k(M))\to H_i(B_{k+1}(M))$ is an isomorphism \cite[Theorem 1.2]{McDuff:CSPNP}.
One says that the configuration spaces of such manifolds exhibit \emph{homological stability}.

One can define such a map on the configuration spaces of a graph, inserting a new particle at the end of a dangling edge, but stability almost never occurs.
In order to remedy this defect, we draw inspiration from a related situation in which homological stability fails, namely that of the \emph{ordered} configuration spaces $\Conf_k(M)$.
Although the Betti numbers of these spaces do not stabilize, the symmetric group representations occurring in homology do.
This phenomenon of \emph{representation stability}, as formalized by Church--Ellenberg--Farb, may be summarized in the statement that the collection $\{H_*(\Conf_k(M))\}_{k\geq0}$ is a \emph{finitely generated module} over a certain combinatorial category \cite[Theorem 6.2.1]{ChurchEllenbergFarb:FIMSRSG}.

In light of these results, a philosophy has emerged that a good notion of homological stability in a given context should be the property of finite generation with respect to some naturally occurring action.
For example, in hindsight, one interprets the theorem of McDuff mentioned above as a statement about the action of the polynomial ring $\mathbb{Z}[\sigma_*]$.

\subsection{Edge stabilization} In this paper, we investigate a new stability phenomenon in the homology $H_*(B(\graf))$, which takes the form of an action by the polynomial ring generated by the edges of $\graf$.
We prove the following analogue of the homological and representation stability enjoyed by configuration spaces of manifolds. 

\begin{theorem}\label{thm:finite generation}
Let $\graf$ be a graph with set of edges $E$.
For any $i\geq0$, the $\mathbb{Z}[E]$-module $H_i(B(\graf);\mathbb{Z})$ is finitely generated.
\end{theorem}
In fact, since finite generation and presentation are equivalent over Noetherian rings, the module is finitely presented. 
By a theorem of Hilbert, Theorem \ref{thm:finite generation} implies eventual polynomial growth of Betti numbers. 
Our next result gives the exact degree of this polynomial in terms of a certain connectivity invariant $\Ramos{\graf}{i}$, which is roughly the largest number of connected components of edges obtainable by removing $i$ essential vertices from $\graf$---see Definition \ref{def:equivalence relation} for a precise definition. In stating the result, we assume that $\graf$ is not a discrete graph, for which the question is trivial.

\begin{theorem}
\label{thm:polynomial growth}
Let $\field$ be a field and $\graf$ a graph with at least one edge. For $k$ sufficiently large, $\dim_\field H_i(B_k(\graf);\field)$ coincides with a polynomial in $k$ of degree $\Ramos{\graf}{i}-1$.
\end{theorem}

First introduced and exploited as a purely algebraic phenomenon by Ramos in \cite[\textsection 3.2]{Ramos:SPHTBG} for trees and by the authors of the present paper~\cite[\textsection 4.2]{AnDrummond-ColeKnudsen:SSGBG} in general, this module structure has a simple topological origin.
The \emph{edge stabilization map} \[\sigma_e:B(\graf)\to B(\graf)\] acts by replacing the subconfiguration of particles on $e$ with the collection of averages of consecutive particles or endpoints---see Figure \ref{fig:edge stabilization} and Section \ref{section:topological edge action} below.

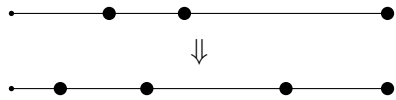
\begin{figure}[ht]
\begin{center}
\begin{tikzpicture}
\fill[black] (0,0) circle (1pt);
\fill[black] (5,0) circle (2.5pt);
\fill[black] (2.3,0) circle (2.5pt);
\fill[black] (1.3,0) circle (2.5pt);
\draw(0,0) -- (5,0);
\draw(2.5,-.2) node[below]{$\Downarrow$};
\begin{scope}[yshift=-1cm]
\fill[black] (0,0) circle (1pt);
\fill[black] (5,0) circle (2.5pt);
\fill[black] (3.65,0) circle (2.5pt);
\fill[black] (1.8,0) circle (2.5pt);
\fill[black] (.65,0) circle (2.5pt);
\draw(0,0) -- (5,0);
\end{scope}
\end{tikzpicture}
\end{center}
\caption{Edge stabilization}\label{fig:edge stabilization}
\end{figure}

\subsection{Edge formality} Because the various edge stabilization maps commute on the nose---not merely up to homotopy---the action of $\mathbb{Z}[E]$ on homology arises from a $\mathbb{Z}[E]$-action on singular chains.
This chain level structure is surprisingly rich; indeed, we show that it almost always carries strictly more information than the action on homology.

\begin{theorem}\label{thm:formality}
Let $\graf$ be a graph with set of edges $E$ and $R$ a commutative ring.
The singular chain complex of $B(\graf)$ with coefficients in $R$ is formal as a differential bigraded $R[E]$-module if and only if $\graf$ is a small graph.
\end{theorem}

Here we declare a graph to be \emph{small} if, after smoothing as many bivalent vertices as possible, no vertex has three distinct edges---see Definition \ref{def:small graph}.
Small graphs are very primitive objects; indeed, each connected component of a small graph is an isolated vertex, an interval $\intervalgraph$, a cycle $\cyclegraph{}$, a lollipop $\lollipopgraph{}$, a figure-eight $\figureeightgraph$, or a handcuff $\handcuffgraph$. See Figure~\ref{figure: small graphs}.
\begin{figure}[ht]
\begin{center}
\begin{tikzpicture}
\begin{scope}[xshift=-2.5cm]
\fill[black] (0,0) circle (2.5pt);
\end{scope}
\begin{scope}[xshift=-1.5cm]
\fill[black] (0,.25) circle (2.5pt);
\fill[black] (0,-.25) circle (2.5pt);
\draw(0,-.25) -- (0,.25);
\draw(0,-1.2) node[below]{$\intervalgraph$};
\end{scope}
\begin{scope}[xshift=-.25cm]
\fill[black] (0,0) circle (2.5pt);
\draw(0,.25) circle (.25 cm);
\draw(0,-1.2) node[below]{$\cyclegraph{}$};
\end{scope}
\begin{scope}[xshift=1.25cm]
\fill[black] (0,0) circle (2.5pt);
\fill[black] (0,-.5) circle (2.5pt);
\draw(0,.25) circle (.25 cm);
\draw(0,-.5) -- (0,0);
\draw(0,-1.2) node[below]{$\lollipopgraph{}$};
\end{scope}
\begin{scope}[xshift=2.75cm]
\fill[black] (0,0) circle (2.5pt);
\draw(0,.25) circle (.25 cm);
\draw(0,-.25) circle (.25 cm);
\draw(0,-1.2) node[below]{$\figureeightgraph$};
\end{scope}
\begin{scope}[xshift=4.75cm]
\fill[black] (-.25,0) circle (2.5pt);
\fill[black] (.25,0) circle (2.5pt);
\draw(.5,0) circle (.25 cm);
\draw(-.5,0) circle (.25 cm);
\draw(-.25,0) -- (.25,0);
\draw(0,-1.2) node[below]{$\handcuffgraph$};
\end{scope}
\end{tikzpicture}
\end{center}
\caption{Homeomorphism types of connected small graphs}
\label{figure: small graphs}
\end{figure}
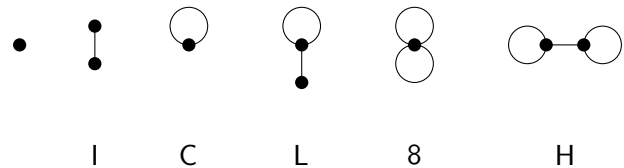

This result is a measure of the inherent complexity of the family of graph braid groups of $\graf$. It also provides an explanation for this complexity, as all non-formal behavior arises from variations on the following simple example. The configuration space of two points in a star graph with three edges is homotopy equivalent to a circle. The fundamental class of this circle is represented by a cycle where the two points orbit one another by taking turns passing through the central vertex. This cycle is decomposable over the ring of edges, but the class itself is indecomposable, and non-formality follows.

\subsection{Techniques and previous work}

We prove our theorems by appealing to a small and explicit chain complex $S(\graf)$ with an action of $\mathbb{Z}[E]$---see Theorem \ref{thm:module iso}---which functorially computes the homology and non-functorially models the singular chains of $B(\graf)$, both as $\mathbb{Z}[E]$-modules. In the proof of Theorem \ref{thm:polynomial growth}, we exploit a family of spectral sequences computing the homology of this complex---see Remark \ref{remark:ss} and Section \ref{section:growth}. These spectral sequences are an expansion of the ``vertex explosion'' technique introduced by the authors in \cite[Definition 5.14]{AnDrummond-ColeKnudsen:SSGBG}, and we expect them to prove useful in future computations.

The \'{S}wi\k{a}tkowski complex $S(\graf)$ first arose as the cellular chain complex of a cubical deformation retract of $B(\graf)$ \cite[\textsection 1]{Swiatkowski:EHDCSG} and was later derived independently by the authors of this paper from ideas involved in factorization homology and discrete Morse theory \cite[\textsection 4.2]{AnDrummond-ColeKnudsen:SSGBG}.
We use features of both approaches to $S(\graf)$, so we must compare them.
This comparison, which appears as Proposition \ref{prop:swiatkowski comparison}, is both a key technical step and a unification of perspectives.

For the case of configuration spaces of trees, Theorem~\ref{thm:polynomial growth} was proven by direct computation of all Betti numbers by Maci\k{a}\.{z}ek and Sawicki~\cite[Equation~(23)]{MaciazekSawicki:HGP1CG} and by Ramos~\cite[Theorem A]{Ramos:SPHTBG}, who also stated the upper bound half of the general case as a conjecture.

Ramos' methods involved a phenomenon similar to edge stabilization. In his work, the action of $\mathbb{Z}[E]$ is purely algebraic, occurring at the level of discrete Morse complexes, and limited to trees. We expect the two actions on homology to coincide.

\subsection{Questions} Our work invites the following questions.
\begin{enumerate}
\item Is there an analogue of edge stabilization for \emph{ordered} configuration spaces of graphs?
\end{enumerate}
L\"utgehetmann and Recio-Mitter recently constructed a stabilization map on ordered configuration spaces of graphs on an edge near an essential vertex~\cite[Proposition 5.6]{LutgehetmannRecioMitter:TCCSFAGBG} using other information at the vertex. 
The nature of the optimal algebraic structure organizing these maps remains unclear.
\begin{enumerate}[resume]
\item Are there analogues of edge stabilization for higher dimensional cells?
\end{enumerate}
One conceptual explanation for edge stabilization is that the unordered configuration spaces of an edge are all contractible, which guarantees that a certain extension problem is unobstructed. 
In the context of both of these questions, the corresponding extension problems are obstructed. 
This fact explains why L\"utgehetmann and Recio-Mitter's stabilization involves additional information beyond the choice of an edge.
It also indicates that stabilization for higher dimensional cells would likely have to involve a similar choice of further local information or, more radically, some kind of non-local invariants.
\begin{enumerate}[resume]
\item What is the significance of the invariant $\Ramos{\graf}{i}$? For example, is its role here connected to its appearance in the ``cut polynomials'' of right-angled Artin groups studied by Papadima and Suciu~\cite[\textsection 4.2]{PapadimaSuciu:AIRAAG}?
\item Is there a simple characterization, extending Theorem \ref{thm:formality} and Proposition \ref{prop:tails formality} below, of the subsets $E_0\subseteq E$ for which formality holds over $\mathbb{Z}[E_0]$?
\item Our proof of Theorem~\ref{theorem: degree theorem} shows that the $W$-tori of Definition~\ref{defi: W-torus} account for a positive fraction of the homology of configuration spaces of graphs---see Remark~\ref{remark: Universal Generators} for further discussion of this point. Can this estimate be improved and/or extended to other families of generators?
\end{enumerate}

\subsection{Linear outline} The paper following the introduction is divided into five sections. 
In Section \ref{section:edge stabilization}, we introduce edge stabilization and the \'{S}wi\k{a}tkowski complex, and we connect the two via the statement of Theorem \ref{thm:module iso}, which implies Theorem \ref{thm:finite generation}. 
Section \ref{section:growth} is concerned with the lower and upper bounds on growth necessary to establish Theorem \ref{thm:polynomial growth}, and Section \ref{section:formality} assembles the proof of Theorem \ref{thm:formality}. Finally, in Sections \ref{section:two isomorphisms} and \ref{section:long ends and proof}, we return to complete the proof of Theorem \ref{thm:module iso}, in the process forging a connection between the cellular models of Abrams and \'{S}wi\k{a}tkowski.

\subsection{Conventions}

Bigradings of modules are by \emph{degree} and \emph{weight}, and all are non-negative.
The braiding isomorphism for a tensor product of modules has a sign which depends on degree and not on weight: if $x$ and $y$ have degree $i$ and $j$, the braiding isomorphism takes $x\otimes y$ to $(-1)^{ij}y\otimes x$. 
We write $[m]$ for the degree shift functor by $m$ and $\{n\}$ for the weight shift functor by $n$ so that the degree $i$ and weight $j$ homogeneous component of $M[m]\{n\}$ is the degree $i-m$ and weight $j-n$ homogeneous component of $M$. 
In a differential graded context, differentials preserve weight.

The category $\Mod$ has objects pairs $(R,M)$ where $R$ is a weight-graded commutative unital ring and $M$ a differential bigraded $R$-module. 
A morphism in $\Mod$ between $(R_1,M_1)$ and $(R_2,M_2)$ is a pair $(f,g)$, where $f$ is a weight-graded ring morphism from $R_1$ to $R_2$ and $g$ is a differential bigraded $R_1$-module morphism from $M_1$ to $f^*M_2$ (that is, $M_2$ with the action $rm\coloneqq f(r)m$).

\subsection*{Acknowledgements}
The first two authors were supported by IBS\nobreakdash-R003\nobreakdash-D1. 
The third author was supported by NSF award 1606422.
The authors would like to thank Daniel L\"utgehetmann, Hyo Won Park, and Eric Ramos for illuminating conversations.

\section{Edge stabilization}\label{section:edge stabilization}

After establishing terminology and notation regarding graphs, we define the edge stabilization map $\sigma_e:B(\graf)\to B(\graf)$, where $e$ is an edge of the graph $\graf$.
At the level of homology, these maps give rise to an action by the polynomial ring generated by the edges of $\graf$.
In Theorem \ref{thm:module iso}, we present a small chain complex computing this homology, together with its module structure.

\subsection{A category of graphs}\label{section:graph conventions} A \emph{graph} $\graf$ is a finite $1$-dimensional CW complex.
Its $0$-cells and open $1$-cells are its \emph{vertices} and \emph{edges}, and the set of such is denoted $V(\graf)$ and $E(\graf)$, or simply $V$ and $E$, respectively.
A \emph{half-edge} is an end of an edge, and the set of such is denoted $H(\graf)$ or simply $H$.
The vertices of an edge $V(e)$ are the vertices contained in the closure of that edge in $\graf$.
The edges of a vertex $E(v)$ are the edges incident to the vertex $v$.
The half-edges of a vertex $H(v)$ (or of an edge $H(e)$) are the half-edges incident to $v$ (contained in $e$).
For $h$ in $H$, we write $v(h)$ and $e(h)$ for the corresponding vertex and edge.

The \emph{valence} of a vertex $v$ is the cardinality of $H(v)$, denoted $d(v)$.
The vertex $v$ is \emph{isolated} if $d(v)=0$ and \emph{essential} if $d(v)\ge 3$.
We shall sometimes write $V^{\geq2}$ and $\essential$ for the set of vertices of valence at least 2 and the set of essential vertices, respectively.
An edge with a $1$-valent vertex is a \emph{tail}.
A \emph{self-loop} at a vertex is an edge whose entire boundary is attached at that vertex.

\begin{example}
The cone on $\{1,\ldots, n\}$ is a graph $\stargraph{n}$ with $n+1$ vertices, with one of valence $n$ and $n$ of valence 1. 
These graphs are called \emph{star graphs} and the cone point the \emph{star vertex}. 
\end{example}

\begin{figure}[ht]
\begin{center}
\begin{tikzpicture}
\begin{scope}[xshift=-2cm]
\fill[black] (0,-.5) circle (2.5pt);
\fill[black] (0,-1) circle (2.5pt);
\draw(0,-1) -- (0,-.5);
\draw(0,-1.2) node[below]{$\stargraph{1}$};
\end{scope}
\fill[black] (0,-.5) circle (2.5pt);
\fill[black] (0,-1) circle (2.5pt);
\fill[black] (0,0) circle (2.5pt);
\draw(0,-1) -- (0,0);
\draw(0,-1.2) node[below]{$\stargraph{2}$};
\begin{scope}[xshift=2.5cm]
\begin{scope}[yshift=-.5cm]
\fill[black] (0,0) circle (2.5pt);
\fill[black] (0,-.5) circle (2.5pt);
\fill[black] (.433,.25) circle (2.5pt);
\fill[black] (-.433,.25) circle (2.5pt);
\draw(0,0) -- (0,-.5);
\draw(0,0) --(.433,.25);
\draw(0,0) --(-.433,.25);
\end{scope}
\draw(0,-1.2) node[below]{$\stargraph{3}$};
\end{scope}
\begin{scope}[xshift=5.5cm]
\begin{scope}[yshift=-.5cm]
\fill[black] (0,0) circle (2.5pt);
\fill[black] (.358,.358) circle (2.5pt);
\fill[black] (.358,-.358) circle (2.5pt);
\fill[black] (-.358,.358) circle (2.5pt);
\fill[black] (-.358,-.358) circle (2.5pt);
\draw(-.358,-.358) -- (.358,.358);
\draw(-.358,.358) -- (.358,-.358);
\end{scope}
\draw(0,-1.2) node[below]{$\stargraph{4}$};
\end{scope}
\end{tikzpicture}
\end{center}
\caption{Star graphs}
\end{figure}
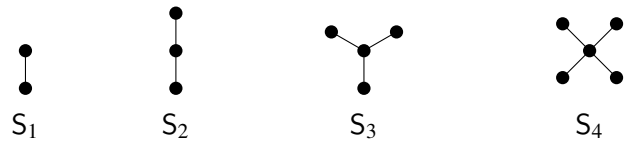

A \emph{parametrization} of a graph $\graf$ is a set of homeomorphisms $D_h:e(h)\to (0,5)$ for $h\in H$ such that
\begin{enumerate}
\item $D_h$ maps the $h$ end of $e(h)$ to the $0$ end of $(0,5)$, and
\item if $h_1\neq h_2\in H(e)$, then $D_{h_2}=5-D_{h_1}$.
\end{enumerate} 
Making a choice of parametrization does not affect the homeomorphism type of configuration spaces.
We will sometimes implicitly identify an edge of a parametrized graph with the interval $(0,5)$.
Up to homotopy, all constructions on parametrized graphs will be independent of the choice of parametrization.

\begin{definition}
Let $f:\graf_1\to \graf_2$ be a continuous map between graphs.
We say that $f$ is a \emph{graph morphism} if 
\begin{enumerate}
\item the inverse image $f^{-1}(V(\graf_2))$ is contained in $V(\graf_1)$ and 
\item the map $f$ is injective.
\end{enumerate}
We call a graph morphism a \emph{smoothing} if it is a homeomorphism and a \emph{graph embedding} if it preserves vertices.
A graph morphism can be factored into a graph embedding followed by a smoothing. At times, we may refer to the inverse of a smoothing as a \emph{subdivision}, and we caution the reader that these are typically not graph morphisms.

The composite of graph morphisms is a graph morphism, and we obtain in this way a category $\Gph$.
Although the objects of $\Gph$ are simply finite 1-dimensional CW complexes, not all morphisms are cellular.
A \emph{subgraph} is the image of a graph embedding.
\end{definition}
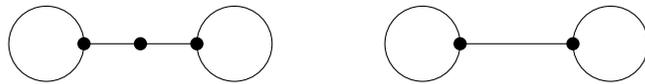
\begin{figure}[ht]
\begin{center}
\begin{tikzpicture}
\fill[black] (0,0) circle (2.5pt);
\fill[black] (1.5,0) circle (2.5pt);
\fill[black] (.75,0) circle (2.5pt);
\draw(0,0) -- (1.5,0);
\draw(-.5,0) circle (.5cm);
\draw(2,0) circle (.5cm);
\begin{scope}[xshift=5cm]
\fill[black] (0,0) circle (2.5pt);
\fill[black] (1.5,0) circle (2.5pt);
\draw(0,0) -- (1.5,0);
\draw(-.5,0) circle (.5cm);
\draw(2,0) circle (.5cm);
\end{scope}
\end{tikzpicture}
\end{center}
\caption{There is a graph morphism (in fact a smoothing) from left to right but not from right to left}
\end{figure}

Since graph morphisms are injective, they induce maps at the level of configuration spaces.
Thus, it is natural to view $H_*(B(-))$ as a functor from the category $\Gph$ to bigraded Abelian groups, where the weight grading records the cardinality of a configuration.
As we will see in the following section, there is more structure to be found.

We close this section with the following pair of definitions premised on our notion of a smoothing.

\begin{definition}\label{def:small graph}
Let $\graf$ be a graph. 
\begin{enumerate}
\item We say that $\graf$ is \emph{smooth} if every smoothing with domain $\graf$ is an isomorphism.
\item We say that $\graf$ is \emph{small} if, in any maximal smoothing of $\graf$, there is no vertex with three distinct edges. Otherwise, $\graf$ is \emph{large}.
\end{enumerate}
\end{definition}

Smoothness is almost, but not quite, equivalent to having no bivalent vertices; indeed, the cycle $\cyclegraph{}$ (see Figure~\ref{figure: small graphs}) is smooth.

\subsection{Topological edge action}\label{section:topological edge action}

We now introduce the promised stabilization.

\begin{definition}\label{def:edge stabilization}
Let $\graf$ be a parametrized graph and $e$ an edge. \emph{Edge stabilization} at $e$ is the map $\sigma_e:B(\graf)\to B(\graf)$ that preserves partial configurations in the complement of $(2,3)\subseteq e$ and replaces the partial configuration $\{x_1\leq \cdots\leq x_j\}\subseteq {(2,3)}$ with
\[
\left\{\frac{2+x_1}{2},\frac{x_1+x_2}{2},\ldots, \frac{x_{j-1}+x_j}{2},\frac{x_j+3}{2}\right\}.
 \] 
\end{definition}

It is a direct verification that the map $\sigma_e$ is continuous and independent of parametrization up to homotopy.
See Figure \ref{fig:edge stabilization redux} for a depiction of this map.

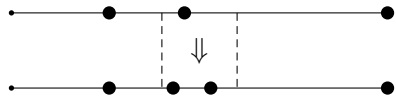
\begin{figure}[ht]
\begin{center}
\begin{tikzpicture}
\fill[black] (0,0) circle (1pt);
\fill[black] (5,0) circle (2.5pt);
\fill[black] (2.3,0) circle (2.5pt);
\fill[black] (1.3,0) circle (2.5pt);
\draw(0,0) -- (5,0);
\draw(2.5,-.2) node[below]{$\Downarrow$};
\draw[densely dashed] (2,0)--(2,-1);
\draw[densely dashed] (3,0)--(3,-1);

\begin{scope}[yshift=-1cm]
\fill[black] (0,0) circle (1pt);
\fill[black] (5,0) circle (2.5pt);
\fill[black] (2.15,0) circle (2.5pt);
\fill[black] (2.65,0) circle (2.5pt);
\fill[black] (1.3,0) circle (2.5pt);
\draw(0,0) -- (5,0);
\end{scope}
\end{tikzpicture}
\end{center}
\caption{Edge stabilization redux}\label{fig:edge stabilization redux}
\end{figure}

\begin{remark}
Parametrizing edges by $(0,5)$ and stabilizing in the subinterval $(2,3)$ are choices designed to interface well with the arguments of Section \ref{section:two isomorphisms}. 
The more obvious choices of parametrizing edges by $(0,1)$ and stabilizing in this entire interval (see Figure~\ref{fig:edge stabilization}) lead to maps that differ only up to homotopy.
\end{remark}

\begin{remark}
Stabilization by adding points to the tails of a graph has been considered~\cite[\textsection 5]{FarleySabalka:OCRTBG}, and addition of points near the boundary of a manifold is a well-studied phenomenon---see~\cite[\textsection 4]{McDuff:CSPNP}, for example.
The existence of stabilization maps at internal edges is new, but see \cite[\textsection 3.2]{Ramos:SPHTBG} for a related algebraic stabilization mechanism in the context of trees.
\end{remark}

Passing to homology, we obtain an action of the weight graded ring $\mathbb{Z}[E]$ on $H_*(B(\graf))$.
This action is natural in the sense that a graph morphism from $\graf_1$ to $\graf_2$ induces a commutative diagram 
\[
\begin{tikzcd}
\mathbb{Z}[E(\graf_1)]\otimes H_*(B(\graf_1))\arrow[d]\arrow[r]&\mathbb{Z}[E(\graf_2)]\otimes H_*(B(\graf_2))\ar[d]\\
H_*(B(\graf_2))\ar[r]&H_*(B(\graf_2)).
\end{tikzcd}
\]
In this way, the homology of configuration spaces of graphs lifts to a functor $H_*(B(-)):\Gph\to \Mod$.

\subsection{The \'{S}wi\k{a}tkowski complex}\label{section:swiatkowski}

We now present a convenient chain model for $H_*(B(\graf))$, thought of as a functor valued in $\Mod$ via edge stabilization.

\begin{definition}\label{def:swiatkowski complex}
Let $\graf$ be a graph and let $R$ be a commutative ring.
For $v\in V$, set $\localstates{v}=\mathbb{Z}\langle \varnothing,v, h\in H(v)\rangle.$
The \emph{\'{S}wi\k{a}tkowski complex} of $\graf$ (with coefficients in $R$) is the $R[E]$-module 
\[\intrinsic{\graf;R}=R[E]\otimes_{\mathbb{Z}}\bigotimes_{v\in V}\localstates{v},\] 
endowed with the bigrading $|\varnothing|=(0,0)$, $|v|=|e|=(0,1)$, and $|h|=(1,1)$, together with the differential $\partial$ determined by setting $\partial(h)=e(h)-v(h).$ 
\end{definition}

Typically, the ring $R$ will be $\mathbb{Z}$ or a field. When the coefficient ring is $\mathbb{Z}$, we omit it from the notation. 
\begin{remark}\label{remark:heuristic}
The following geometric heuristic is often helpful in dealing with the \'{S}wi\k{a}tkowski complex. 
\begin{enumerate}
\item We think of a generator of $\localstates{v}$ as prescribing the local ``state'' of a configuration near $v$. Specifically, the generator $\varnothing$ represents the absence of a particle at $v$, the generator $v$ represents the presence of a stationary particle at $v$, and the generator $h\in H(v)$ represents an infinitesimal path of a particle exiting $v$ in the direction of $h$. 
See Figure~\ref{figure: local states}.
\item We think of a monomial in $\mathbb{Z}[E]$ as prescribing the state of a configuration on the edges of $\graf$; for example, the generator $e_1e_2^2$ represents the presence of a stationary particle on $e_1$ and two stationary particles on $e_2$. 
See Figure~\ref{figure: local states}.
\item A basis element of the \'{S}wi\k{a}tkowski complex is a prescription of a global state, which is to say that we choose a local state from each $\localstates{v}$ and populate the edges with a monomial in $E$. 
\item The first grading is the natural homological grading, since $h\in H$ is a path and all other generators involve stationary particles, and the second grading is given by counting the number of particles. 
\item The differential takes a path to the difference of its endpoints.
\end{enumerate}
\end{remark}
\begin{figure}[ht]
\begin{center}
\begin{tikzpicture}
\begin{scope}
\fill[black] (0,0) circle (1pt);
\draw(-.707,-.707) -- (.707,.707);
\draw(-.707,.707) -- (.707,-.707);
\draw(0,-1) node[below]{$\varnothing$};
\end{scope}
\begin{scope}[xshift=3cm]
\fill[black] (0,0) circle (2.5pt);
\draw(-.707,-.707) -- (.707,.707);
\draw(-.707,.707) -- (.707,-.707);
\draw(0,-1) node[below]{$v$};
\end{scope}
\begin{scope}[xshift=6cm]
\fill[black] (0,0) circle (2.5pt);
\draw(-.707,-.707) -- (.707,.707);
\draw[line width=2] (0,0) --(.4,-.4);
\fill[black] (.4,-.4) circle (2.5pt);
\draw[->] (.175,.025)--(.425,-.225);
\draw(-.707,.707) -- (.707,-.707);
\draw(0,-1) node[below]{$h$};
\end{scope}
\begin{scope}[xshift=9cm]
\draw(-.707,.707) -- (.707,.707);
\fill[black] (0,.707) circle (2.5pt);
\draw(-.707,-.707) -- (.707,-.707);
\fill[black] (.472,-.707) circle (2.5pt);
\fill[black] (-.472,-.707) circle (2.5pt);
\draw(0,.707) node[below]{$e_1$};
\draw(0,-.707) node[below]{$e_2^2$};
\end{scope}
\end{tikzpicture}
\end{center}
\caption{Local states near a vertex $v$ and the state $e_1e_2^2$}
\label{figure: local states}
\end{figure}
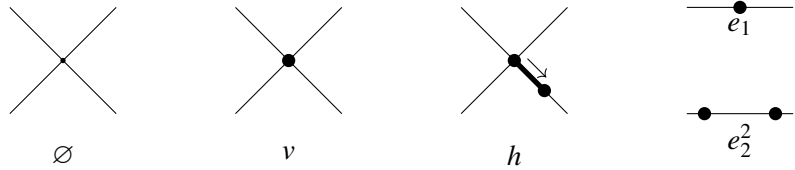

\begin{notation}
We systematically omit all factors of $\varnothing$ and all tensor symbols from the notation, and we regard half-edge generators at different vertices as permutable up to sign. Thus, if $\graf$ is the interval with vertices $0$ and $1$, half-edges $h_0$ and $h_1$, and edge $e$, then we write \begin{align*}
e^3\otimes \varnothing\otimes \varnothing &=e^3\\
e^3\otimes \varnothing\otimes h_2 &=e^3h_2\\
e^3\otimes h_1\otimes h_2 &=e^3h_1h_2=-e^3h_2h_1.
\end{align*}
See Figure~\ref{figure: interval example}.
\begin{figure}[ht]
\begin{center}
\begin{tikzpicture}
\begin{scope}
\fill[black] (-2.5,0) circle (1pt);
\fill[black] (2.5,0) circle (1pt);
\fill[black] (-.5,0) circle (2.5pt);
\fill[black] (0,0) circle (2.5pt);
\fill[black] (.5,0) circle (2.5pt);
\draw(-2.5,0) -- (2.5,0);
\draw(0,0) node[below]{$e^3$};
\end{scope}
\begin{scope}[yshift=-1cm]
\fill[black] (-2.5,0) circle (1pt);
\fill[black] (2.5,0) circle (1pt);
\fill[black] (-.5,0) circle (2.5pt);
\fill[black] (0,0) circle (2.5pt);
\fill[black] (.5,0) circle (2.5pt);
\fill[black] (1.5,0) circle (2.5pt);
\fill[black] (2.5,0) circle (2.5pt);
\draw(-2.5,0) -- (2.5,0);
\draw[line width=2](1.5,0) -- (2.5,0);
\draw(0,0) node[below]{$e^3h_2$};
\end{scope}
\begin{scope}[yshift=-2cm]
\fill[black] (-2.5,0) circle (1pt);
\fill[black] (2.5,0) circle (1pt);
\fill[black] (-.5,0) circle (2.5pt);
\fill[black] (0,0) circle (2.5pt);
\fill[black] (.5,0) circle (2.5pt);
\fill[black] (1.5,0) circle (2.5pt);
\fill[black] (2.5,0) circle (2.5pt);
\fill[black] (-1.5,0) circle (2.5pt);
\fill[black] (-2.5,0) circle (2.5pt);
\draw(-2.5,0) -- (2.5,0);
\draw[line width=2](1.5,0) -- (2.5,0);
\draw[line width=2](-1.5,0) -- (-2.5,0);
\draw(0,0) node[below]{$e^3h_1h_2$};
\end{scope}
\end{tikzpicture}
\end{center}
\caption{Three global states on the interval}
\label{figure: interval example}
\end{figure}
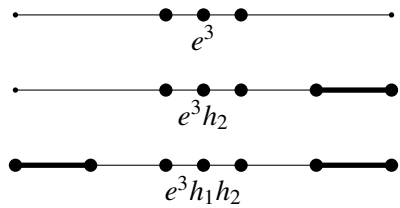
\end{notation}

A graph morphism $f:\graf_1\to\graf_2$ determines a map $\intrinsic{f;R}:\intrinsic{\graf_1;R}\to \intrinsic{\graf_2;R}$ as follows~\cite[\textsection 4.2]{AnDrummond-ColeKnudsen:SSGBG}: an edge is sent to its image under $f$; for a vertex $v$ such that $f(v)$ is also a vertex, there is an evident map $\localstates{v}\to \localstates{f(v)}$; and, for a vertex such that $f(v)$ is an edge, we instead use the map $\localstates{v}\to\mathbb{Z}[f(v)]$ sending $\varnothing$ to $1$, $v$ to $e$, and half-edges to $0$. 
This map respects the bigrading, differential, and module structures, so we obtain a functor $\intrinsic{-;R}:\Gph\to\Mod$.
Since $\partial$ is $R[E]$-linear by definition, this module structure descends to homology. 

\begin{theorem}\label{thm:module iso}
There is a natural isomorphism \[H_*(B(\graf);R)\cong H_*(\intrinsic{\graf;R})\] of functors from $\Gph$ to $\Mod$.
\end{theorem}

At the level of bigraded Abelian groups, this natural isomorphism was established in \cite[Theorem 4.5]{AnDrummond-ColeKnudsen:SSGBG}---see Theorem \ref{thm:ADCK} below.
To conclude the full statement over $\mathbb{Z}$, we will check that this isomorphism is compatible with the respective $\mathbb{Z}[E]$-actions, a task which we take up in Section \ref{section:two isomorphisms} below. The general case follows by the universal coefficients theorem.

Since $\intrinsic{\graf}$ is finitely generated over $\mathbb{Z}[E]$ by definition, and since $\mathbb{Z}[E]$ is Noetherian, Theorem \ref{thm:finite generation} follows immediately from Theorem~\ref{thm:module iso}. 

We close this section with an introduction of a smaller variant of the \'{S}wi\k{a}tkowski complex, which is often more convenient.

\begin{definition}\label{definition: reduced intrinsic definition}
Let $\graf$ be a graph.
For each $v\in V$, let $\reducedlocalstates{v}\subseteq \localstates{v}$ be the subspace spanned by $\varnothing$ and the differences $h_{ij}\coloneqq{}h_i-h_j$ of half-edges.
The \emph{reduced \'{S}wi\k{a}tkowski complex} with coefficients in $R$ is \[\fullyreduced{\graf;R}\coloneqq R[E]\otimes_{\mathbb{Z}} \bigotimes_{v\in V}\reducedlocalstates{v},\] considered as a subcomplex and submodule of $\intrinsic{\graf;R}$.
To be explicit, the differential is determined by $\partial(h_{ij})=e(h_i)-e(h_j)$.
\end{definition}

The inclusion $\fullyreduced{\graf;R}\subseteq S(\graf;R)$ is an $R[E]$-linear quasi-isomorphism. The $R=\mathbb{Z}$ case is~\cite[Proposition 4.9]{AnDrummond-ColeKnudsen:SSGBG}, and the general case follows by the universal coefficients theorem.
Note that $\fullyreduced{-;R}$ is functorial for graph morphisms. 

\subsection{Vertex explosion and star classes}\label{section:tools}
In this section, we review some tools from \cite[\textsection 5]{AnDrummond-ColeKnudsen:SSGBG} afforded by the \'{S}wi\k{a}tkowski complex, which will play an important role in the proofs of Theorems \ref{thm:polynomial growth} and \ref{thm:formality}. The first of these is an exact sequence which is useful in reducing computations of $H_*(B(\graf))$ to computations for simpler graphs. 

\begin{definition}
For $v\in V$, we write $\graf_v$ for the \emph{vertex explosion} of $\graf$ at $v$, which is the graph obtained by 
\begin{enumerate}
\item replacing the vertex $v$ with $\{v\}\times H(v)$ and 
\item modifying the attaching maps for half-edges at $v$ by attaching $h$ to $(v,h)$.
\end{enumerate}
\end{definition}
There is a graph morphism from $\graf_v$ to $\graf$ which takes each edge to itself, takes the vertex $(v, h)$ to $e(h)$, and takes each other vertex to itself.
Defining this morphism requires choices of precisely where in $e(h)$ to send $(v,h)$, but the isotopy class of this graph morphism is unique.
See Figure~\ref{figure: vertex explosion}.
\begin{figure}[ht]
\begin{center}
\begin{tikzpicture}
\begin{scope}
\fill[black] (0,0) circle (2.5pt);
\draw(-.707,-.707) -- (.707,.707);
\draw(-.707,.707) -- (.707,-.707);
\draw(0,-1.2) node[below]{$\graf$};
\draw(2,0) node{$\xleftarrow{\text{smoothing}}$};
\draw(6,0) node{$\xleftarrow{\text{graph embedding}}$};
\end{scope}
\begin{scope}[xshift=4cm]
\fill[black] (0,0) circle (2.5pt);
\fill[black] (-.3,-.3) circle (2.5pt);
\fill[black] (-.3,.3) circle (2.5pt);
\fill[black] (.3,-.3) circle (2.5pt);
\fill[black] (.3,.3) circle (2.5pt);
\draw(-.707,-.707) -- (.707,.707);
\draw(-.707,.707) -- (.707,-.707);
\end{scope}
\begin{scope}[xshift=8cm]
\fill[black] (-.3,-.3) circle (2.5pt);
\fill[black] (-.3,.3) circle (2.5pt);
\fill[black] (.3,-.3) circle (2.5pt);
\fill[black] (.3,.3) circle (2.5pt);
\draw(-.707,-.707) -- (-.3,-.3);
\draw(.707,.707) -- (.3,.3);
\draw(-.707,.707) -- (-.3,.3);
\draw(.707,-.707) -- (.3,-.3);
\draw(0,-1.2) node[below]{$\graf_v$};
\end{scope}
\end{tikzpicture}
\end{center}
\caption{A local picture of vertex explosion along with an intermediate graph which admits a graph morphism from $\graf_v$ and a smoothing to $\graf$}
\label{figure: vertex explosion}
\end{figure}
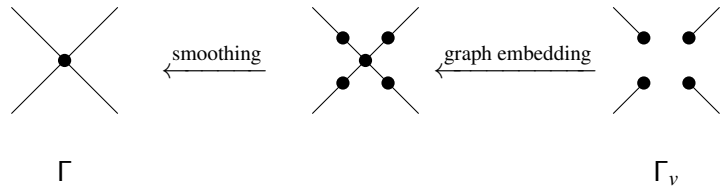

\begin{proposition}[{\cite[Corollary 5.16]{AnDrummond-ColeKnudsen:SSGBG}}]\label{proposition: les vertex reduced}
Fix a half-edge $h_0\in H(v)$.
There is a short exact sequence of differential bigraded $R[E]$-modules
\[
0\to \fullyreduced{\graf_v;R}\to \fullyreduced{\graf;R}\xrightarrow{\psi} \bigoplus_{h\in H(v)\setminus\{h_0\}}\fullyreduced{\graf_v;R}[1]\{1\}\to 0,
\] where $\psi$ sends an element of $\fullyreduced{\graf;R}$, written uniquely as $\beta+\sum_{h\in H(v)\setminus \{h_0\}}(h-h_0)\alpha_h$ with $\beta$ involving no half-edges incident on $v$, to $(\alpha_h)_{h\in H(v)\setminus \{h_0\}}$. 
\end{proposition}
The result in~\cite{AnDrummond-ColeKnudsen:SSGBG} is only taken with $\mathbb{Z}$ coefficients. 
Since $\fullyreduced{\graf_v}$ is degreewise flat, tensoring with $R$ preserves exactness.

In the homology long exact sequence corresponding to this short exact sequence of chain complexes, the connecting homomorphism $\delta$ from $\bigoplus H_*(\fullyreduced{\graf_v;R})\{1\}\to H_*(\fullyreduced{\graf_v;R})$ is given by the formula
\[
\delta \beta_h = (e(h)-e(h_0))\beta_h.\]

\begin{remark}\label{remark:ss}
This exact sequence is a degenerate example of a spectral sequence interpolating between the homology groups of the configuration spaces of $\graf$ and those of the graph obtained by exploding a specified collection of vertices. This type of spectral sequence will play an important role in Section \ref{section:growth} below.
\end{remark}

The second tool is a type of atomic homology class. Recall that the configuration space of two points in the star $\stargraph{3}$ is homotopy equivalent to a circle.

\begin{definition}
\label{defi: star classes}
\cite[\textsection 5.1]{AnDrummond-ColeKnudsen:SSGBG}
Let $\stargraph{3}\to\graf$ be a graph morphism. A {\em star class} is a class in $H_1(B_2(\graf);R)$ which is the image under the induced morphism of the generator $\alpha$ in $H_1(B_2(\graf);R)\simeq R$ represented by the chain $a\in S(\stargraph{3};R)$
\[
a\coloneqq e_1(h_2-h_3)+e_2(h_3-h_1)+e_3(h_1-h_2),
\]
where $h_i$ is the half-edge of $e_i$ adjacent to the star vertex of $\stargraph{3}$.
\end{definition}
We refer to the chain $a$, and to its image, as the \emph{standard representative} of $\alpha$. See Figure~\ref{figure: star class}.
Given $m$ graph morphisms $\stargraph{3}\to \graf$ with pairwise disjoint images, we obtain a degree $m$ homology class, called the \emph{external product} of the star classes \cite[Definition 5.10]{AnDrummond-ColeKnudsen:SSGBG}, whose standard representative is the tensor product of the standard representatives of the factors.

\begin{figure}[ht]
\begin{center}
\begin{tikzpicture}
\draw (1,.125) node{$-$}; 
\draw (3,.125) node{$+$};
\draw (5,.125) node{$-$};
\draw (7,.125) node{$+$};
\draw (9,.125) node{$-$};
\draw [
    thick,
    decoration={
        brace,
        mirror,
        raise=0.05cm
    },
    decorate
] (-.6,-.5) -- (2.6,-.5);
\draw [
    thick,
    decoration={
        brace,
        mirror,
        raise=0.05cm
    },
    decorate
] (3.4,-.5) -- (6.6,-.5);
\draw [
    thick,
    decoration={
        brace,
        mirror,
        raise=0.05cm
    },
    decorate
] (7.4,-.5) -- (10.6,-.5);
\begin{scope}[xshift=4cm]
\draw(0,.6) node [above] {$e_2h_3$};
\draw[black, fill=black] (0,0) circle (1.5pt);
 \draw(0,0) -- (0,.5);
 \draw(0,0) --(.433,-.25);
 \draw(0,0) --(-.433,-.25);
\draw [fill] (0,0) circle [radius = 2pt];
\draw [fill] (.433,-.25) circle [radius = 2pt];
\draw [black, very thick] (0,0)-- (.433,-.25);
\draw [fill] (0,.5) circle [radius = 2pt];
\end{scope}
\begin{scope}[xshift=0cm]
\draw(0,.6) node [above] {$e_3h_1$};
\draw[black, fill=black] (0,0) circle (1.5pt);
 \draw(0,0) -- (0,.5);
 \draw(0,0) --(.433,-.25);
 \draw(0,0) --(-.433,-.25);
\draw [fill] (0,0) circle [radius = 2pt];
\draw [fill] (.433,-.25) circle [radius = 2pt];
\draw [black, very thick] (0,0)-- (-.433,-.25);
\draw [fill] (-.433,-.25) circle [radius = 2pt];
\end{scope}
\begin{scope}[xshift=8cm]
\draw(0,.6) node [above] {$e_1h_2$};
 \draw[black, fill=black] (0,0) circle (1.5pt);
 \draw(0,0) -- (0,.5);
 \draw(0,0) --(.433,-.25);
 \draw(0,0) --(-.433,-.25);
\draw [fill] (0,0) circle [radius = 2pt];
\draw [fill] (-.433,-.25) circle [radius = 2pt];
\draw [black, very thick] (0,0)-- (0,.5);
\draw [fill] (0,.5) circle [radius = 2pt];
\end{scope}
\begin{scope}[xshift=10cm]
\draw(0,.6) node [above] {$e_1h_3$};
\draw[black, fill=black] (0,0) circle (1.5pt);
 \draw(0,0) -- (0,.5);
 \draw(0,0) --(.433,-.25);
 \draw(0,0) --(-.433,-.25);
\draw [fill] (0,0) circle [radius = 2pt];
\draw [fill] (.433,-.25) circle [radius = 2pt];
\draw [black, very thick] (0,0)-- (.433,-.25);
\draw [fill] (-.433,-.25) circle [radius = 2pt];
\end{scope}
\begin{scope}[xshift=2cm]
\draw(0,.6) node [above] {$e_3h_2$};
 \draw[black, fill=black] (0,0) circle (1.5pt);
 \draw(0,0) -- (0,.5);
 \draw(0,0) --(.433,-.25);
 \draw(0,0) --(-.433,-.25);
\draw [fill] (0,0) circle [radius = 2pt];
\draw [fill] (.433,-.25) circle [radius = 2pt];
\draw [black, very thick] (0,0)-- (0,.5);
\draw [fill] (0,.5) circle [radius = 2pt];
\end{scope}
\begin{scope}[xshift=6cm]
\draw(0,.6) node [above] {$e_2h_1$};
\draw[black, fill=black] (0,0) circle (1.5pt);
 \draw(0,0) -- (0,.5);
 \draw(0,0) --(.433,-.25);
 \draw(0,0) --(-.433,-.25);
\draw [fill] (0,0) circle [radius = 2pt];
\draw [fill] (0,.5) circle [radius = 2pt];
\draw [black, very thick] (0,0)-- (-.433,-.25);
\draw [fill] (-.433,-.25) circle [radius = 2pt];
\end{scope}
\begin{scope}[yshift=-1.5cm, xshift=9cm]
\draw(0,0) -- (0,.5);
\draw(0,0) --(.433,-.25);
\draw(0,0) --(-.433,-.25);
\draw[->] (.533,-.0768) .. controls (.1732,.1) .. (.2,.5);
\draw [fill] (0,.5) circle [radius = 2pt];
\draw [fill] (-.433,-.25) circle [radius = 2pt];
\draw [fill] (.433,-.25) circle [radius = 2pt];
\draw [black, very thick] (0,0)-- (.433,-.25);
\draw [black, very thick] (0,0)-- (0,.5);
\end{scope}
\begin{scope}[yshift=-1.5cm, xshift=5cm]
\draw(0,0) -- (0,.5);
\draw(0,0) --(.433,-.25);
\draw(0,0) --(-.433,-.25);
\draw[->](-.333,-.4232) .. controls (0,-.2) .. (.333,-.4232);
\draw [fill] (0,.5) circle [radius = 2pt];
\draw [fill] (-.433,-.25) circle [radius = 2pt];
\draw [fill] (.433,-.25) circle [radius = 2pt];
\draw [black, very thick] (0,0)-- (-.433,-.25);
\draw [black, very thick] (0,0)-- (.433,-.25);
\end{scope}
\begin{scope}[yshift=-1.5cm, xshift=1cm]
\draw(0,0) -- (0,.5);
\draw(0,0) --(.433,-.25);
\draw(0,0) --(-.433,-.25);
\draw[->](-.2,.5) .. controls (-.1732,.1) .. (-.533,-.0768);
\draw [fill] (0,.5) circle [radius = 2pt];
\draw [fill] (-.433,-.25) circle [radius = 2pt];
\draw [fill] (.433,-.25) circle [radius = 2pt];
\draw [black, very thick] (0,0)-- (-.433,-.25);
\draw [black, very thick] (0,0)-- (0,.5);
\end{scope}
\end{tikzpicture}
\end{center}
\caption{The standard representative of a star class}
\label{figure: star class}
\end{figure}
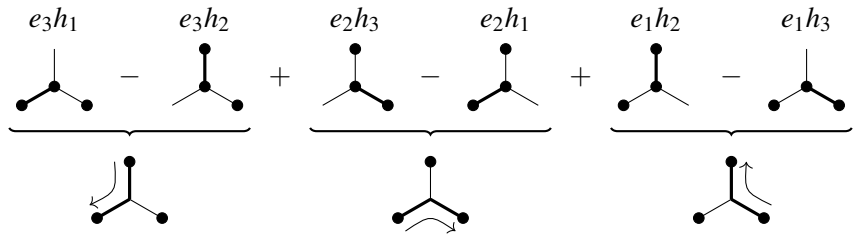

\section{Growth of Betti numbers}\label{section:growth}

Theorem \ref{thm:finite generation} implies that the $\field$-Betti numbers of $B_k(\graf)$ are eventually given by a polynomial in $k$ for any field $\field$ (see, eg,~\cite[Theorem 1.11]{Eisenbud:CAWVTAG}). 
In this section, we determine the exact degree of this polynomial.

\subsection{Connectivity and growth}
We write $\growthdeg{\graf}{i}{\field}$ for the degree of eventual polynomial growth of the $\field$-dimension of $H_i(B_k(\graf);\field)$ in the weight $k$. 
By convention, $\growthdeg{\graf}{i}{\field}=-\infty$ if $H_i(B_k(\graf);\field)=0$ for $k\gg 0$. 

This degree of growth is controlled by a certain elementary connectivity invariant. In order to define this invariant, we introduce the following equivalence relation.

\begin{definition}\label{def:equivalence relation}
Let $\graf$ be a smooth graph. A subset $W\subseteq V^{\geq2}$ determines an equivalence relation ${\sim}_W$ on the edges of $\graf$, where $e{\sim}_W e'$
if and only if $[e]=[e']$ in $\pi_0(\graf\setminus W)$. Writing $E_W=E/{\sim}_W$ for the set of equivalence classes, we define
\begin{align*}
\Ramos{\graf}{W}&=\left| E_W\right|,&
\Ramos{\graf}{i}&=\max \Ramos{\graf}{W}.
\end{align*}
where the maximum is taken over subsets $W$ of cardinality $i$ of $V^{\geq2}$.
We use the convention that $\Ramos{\graf}{i}=-\infty$ if $i>|V^{\geq2}|$. If $\graf$ is not smooth, define $\Ramos{\graf}{i}$ as $\Ramos{\graf'}{i}$ where $\graf'$ is a smooth graph homeomorphic to $\graf$.
\end{definition}
See Figure~\ref{figure: W equivalence}.

\begin{figure}[ht]
\begin{center}
\begin{tikzpicture}
\begin{scope}[xshift=4cm]
\draw[black, fill=black] (-.5,-.5) circle (1.5pt);
\draw[black, fill=black] (.5,-.5) circle (1.5pt);
\draw[black, fill=black] (.5,-1.5) circle (1.5pt);
\draw[black, fill=black] (.5,0) circle (1.5pt);
\draw[black, fill=black] (.5,.5) circle (1.5pt);
\draw[black, fill=black] (1,0) circle (1.5pt);
\draw[black, fill=black] (1,-2) circle (1.5pt);
\draw[black, fill=black] (-.5,-2) circle (1.5pt);
\draw[black, fill=black] (.5,-2.5) circle (1.5pt);
\draw(0,.25) circle (.25cm);
\draw(-.5,-2)--(1,-2);
\draw(1,0)--(0,0)--(-.5,-.5)--(.5,-1.5);
\draw(0,-1)--(.5,-.5);
\draw(.5,0)--(.5,.5);
\draw(0,-2.5)--(.5,-2.5);
\draw(0,-1)--(0,-2);
\draw (.5,-1.5) to[bend right=-30] (.5,-.5);
\draw (.5,-1.5) to[bend right=30] (.5,-.5);
\draw (1,-2) to[bend right=20] (1,0);
\draw (1,-2) to[bend right=-20] (1,0);
\draw[red, fill=red] (0,0) circle (1.5pt);
\draw[red, fill=red] (0,-1) circle (1.5pt);
\draw[red, fill=red] (0,-2) circle (1.5pt);
\draw[red, fill=red] (0,-2.5) circle (1.5pt);
\end{scope}
\begin{scope}[xshift=7cm]
\draw[black, fill=black] (-.5,-.5) circle (1.5pt);
\draw[black, fill=black] (.5,-.5) circle (1.5pt);
\draw[black, fill=black] (.5,-1.5) circle (1.5pt);
\draw[black, fill=black] (.5,0) circle (1.5pt);
\draw[black, fill=black] (.5,.5) circle (1.5pt);
\draw[black, fill=black] (1,0) circle (1.5pt);
\draw[black, fill=black] (1,-2) circle (1.5pt);
\draw[black, fill=black] (-.5,-2) circle (1.5pt);
\draw[black, fill=black] (.5,-2.5) circle (1.5pt);
\draw(0,.25) circle (.25cm);
\draw(-.5,-2)--(1,-2);
\draw(1,0)--(0,0)--(-.5,-.5)--(.5,-1.5);
\draw(0,-1)--(.5,-.5);
\draw(.5,0)--(.5,.5);
\draw(0,-2.5)--(.5,-2.5);
\draw(0,-1)--(0,-2);
\draw (.5,-1.5) to[bend right=-30] (.5,-.5);
\draw (.5,-1.5) to[bend right=30] (.5,-.5);
\draw (1,-2) to[bend right=20] (1,0);
\draw (1,-2) to[bend right=-20] (1,0);
\draw[white, fill=white] (0,0) circle (6pt);
\draw[white, fill=white] (0,-1) circle (6pt);
\draw[white, fill=white] (0,-2) circle (6pt);
\draw[white, fill=white] (0,-2.5) circle (6pt);
\end{scope}
\end{tikzpicture}
\end{center}
\caption{A graph with a choice of $W$ in red, and the seven equivalence classes of edges under $\sim_W$}
\label{figure: W equivalence}
\end{figure}
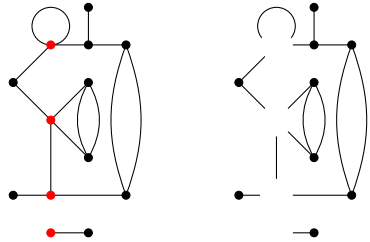

\begin{remark}
Although we work with a different set of conventions, this invariant was essentially defined by Ramos~\cite[page~2306]{Ramos:SPHTBG}, who conjectured Proposition \ref{proposition: degree upper bound} below.
\end{remark}

\begin{example}
\label{example:Ramos calculations}
Here are some basic examples of the behavior of the invariant $\Ramos{\graf}{i}$.
\begin{enumerate}
\item If $\graf$ is an isolated vertex, then $\Ramos{\graf}{0}=0$.
\item If $\graf$ is an isolated edge, then $\Ramos{\graf}{0}=1$.
\item If $\graf$ is a cycle, then $\Ramos{\graf}{0}=\Ramos{\graf}{1}=1$.
\end{enumerate}
In these examples, $\Ramos{\graf}{i}=-\infty$ for all other values of $i$.
\begin{enumerate}[resume]
\item For $\graf=\graf_1\sqcup \graf_2$ a disjoint union of (not necessarily connected) graphs, 
 \[
\Ramos{\graf}{i}=\max_{a+b=i}(\Ramos{\graf_1}{a} + \Ramos{\graf_2}{b}).
 \]
\end{enumerate}
\end{example}

We recall the statement of Theorem \ref{thm:polynomial growth}.

\begin{degreetheorem}
\label{theorem: degree theorem}
If $\graf$ has at least one edge, then $\growthdeg{\graf}{i}{\field}=\Ramos{\graf}{i}-1$.
\end{degreetheorem}

\begin{remark}
A statement valid for an arbitrary graph is available for the sum 
\[\sum_{\ell=0}^k\dim_\field H_i(B_\ell(\graf);\field),\] 
which exhibits eventual polynomial growth of degree $\Ramos{\graf}{i}$.
\end{remark}

\begin{example}
For $n\ge 3$, let $\thetagraph{n}$ denote the Theta-$n$ graph, which is the double cone on $n$ points, as shown:
\[
\thetagraph{n}\coloneqq
\begin{tikzpicture}[baseline=-0.5ex]
\fill[black] (-.5,0) circle (2.5pt); 
\fill[black] (.5,0) circle (2.5pt); 
\draw(0,0) circle (.5cm);
\draw(.5,0) arc (0:180:0.5 and 0.25);
\draw(0,-0.15) node {$\vdots$};
\draw(.5,0)--(-.5,0);
\end{tikzpicture}
\]
It is known that 
\begin{align*}
\dim_\field H_i(B_k(\thetagraph{n});\field)=\begin{cases}
\displaystyle1 & i=0;\\
\displaystyle\binom{n}{2} & i=1;\\
\displaystyle
\binom{n}{2} -1 + \sum_{j=0}^2(1-n)^j\binom{2}{j}\binom{k-j+n-1}{n-1} & i=2.
\end{cases}
\end{align*}
Calculating the zeroth Betti number is trivial, the first is given in \cite[Lemma~3.14]{KoPark:CGBG}, and the second follows from these and knowledge of the Euler characteristic for $B_k(\thetagraph{n})$, which can be extracted from~\cite[Theorem~2]{Gal:ECCSC} and is given explicitly in~\cite[Corollary~5.11]{AnDrummond-ColeKnudsen:SSGBG}. 
These Betti numbers exhibit polynomial growth of degree $0$ for $i\le 1$ and $n-1$ for $i=2$, and these numbers
are one less than respective numbers of components obtained by removing $i$ vertices. Thus, the theorem holds in this example.
\end{example}

\begin{example}
For the complete graph $\completegraph{4}$, the explicit formula for the $i$th Betti number $\dim_\field H_i(B_k(\completegraph{4});\field)$ in terms of $k$ can be computed directly from \cite[\textsection 5.6]{AnDrummond-ColeKnudsen:SSGBG}. Table~\ref{table:complete four-graph} shows the growth $\growthdeg{\completegraph{4}}{i}{\field}$ of the Betti number and the invariant $\Ramos{\completegraph{4}}{i}$, verifying Theorem~\ref{thm:polynomial growth} for $\completegraph{4}$.
\end{example}

\begin{table}[ht]
\[
\renewcommand{\arraystretch}{3.365}
\setlength{\arraycolsep}{10pt}
\begin{array}{rcccc}
i & \dim_\field H_i (B_k(\completegraph{4});\field) & \growthdeg{\completegraph{4}}{i}{\field} & \Ramos{\completegraph{4}}{i} & (\completegraph{4})_W
\\\hline
0 & 1 & 0 & 1 &
\begin{tikzpicture}[baseline=-.5ex,scale=0.7]
\draw(0,0) -- (0,1) -- (.866,-.5) -- (-.866,-.5) -- (0,1);
\draw(0,0) --(.866,-.5);
\draw(0,0) --(-.866,-.5);
\draw[black, fill=black] (0,0) circle (2.5pt);
\draw[black, fill=black] (0,1) circle (2.5pt);
\draw[black, fill=black] (.866,-.5) circle (2.5pt);
\draw[black, fill=black] (-.866,-.5) circle (2.5pt);
\end{tikzpicture}
\\
1 & 4 \text{ (for }k>1\text{)} & 0 & 1 &
\begin{tikzpicture}[baseline=-.5ex,scale=0.7]
\draw(0,0) -- (0,1) -- (.866,-.5) -- (-.866,-.5) -- (0,1);
\draw(0,0) --(.866,-.5);
\draw(0,0) --(-.866,-.5);
\draw[white, fill=white] (0,0) circle (2.5pt);
\draw[black, fill=black] (0,1) circle (2.5pt);
\draw[black, fill=black] (.866,-.5) circle (2.5pt);
\draw[black, fill=black] (-.866,-.5) circle (2.5pt);
\end{tikzpicture}
\\
2 & 6k-15\text{ (for }k>2\text{)} & 1 & 2 &
\begin{tikzpicture}[baseline=-.5ex,scale=0.7]
\draw(0,0) -- (0,1) -- (.866,-.5) -- (-.866,-.5) -- (0,1);
\draw(0,0) --(.866,-.5);
\draw(0,0) --(-.866,-.5);
\draw[white, fill=white] (0,0) circle (2.5pt);
\draw[white, fill=white] (0,1) circle (2.5pt);
\draw[black, fill=black] (.866,-.5) circle (2.5pt);
\draw[black, fill=black] (-.866,-.5) circle (2.5pt);
\end{tikzpicture}
\\
3 & 4\binom{k-3}{3} & 3 & 4 &
\begin{tikzpicture}[baseline=-.5ex,scale=0.7]
\draw(0,0) -- (0,1) -- (.866,-.5) -- (-.866,-.5) -- (0,1);
\draw(0,0) --(.866,-.5);
\draw(0,0) --(-.866,-.5);
\draw[white, fill=white] (0,0) circle (2.5pt);
\draw[white, fill=white] (0,1) circle (2.5pt);
\draw[white, fill=white] (.866,-.5) circle (2.5pt);
\draw[black, fill=black] (-.866,-.5) circle (2.5pt);
\end{tikzpicture}
\\
4 & \binom{k-3}{5} & 5 & 6 &
\begin{tikzpicture}[baseline=-.5ex,scale=0.7]
\draw(0,0) -- (0,1) -- (.866,-.5) -- (-.866,-.5) -- (0,1);
\draw(0,0) --(.866,-.5);
\draw(0,0) --(-.866,-.5);
\draw[white, fill=white] (0,0) circle (2.5pt);
\draw[white, fill=white] (0,1) circle (2.5pt);
\draw[white, fill=white] (.866,-.5) circle (2.5pt);
\draw[white, fill=white] (-.866,-.5) circle (2.5pt);
\end{tikzpicture}
\\
\ge 5 & 0 & -\infty & -\infty
\end{array}
\]
\caption{For the complete graph $\completegraph{4}$, a table containing the Betti number $\dim_\field H_i (B_k(\completegraph{4});\field)$, the growth rate $\growthdeg{\completegraph{4}}{i}{\field}$, the invariant $\Ramos{\completegraph{4}}{i}$, and the corresponding complement $(\completegraph{4})_W$}
\label{table:complete four-graph}
\end{table}

\begin{definition}
We say that a graph $\graf$ is \emph{normal} if it is connected, smooth, and has an essential vertex.
\end{definition}

The bulk of the theorem is contained in the following two results, whose proofs will occupy the rest of the section. The second is essentially Conjecture 4.3 of \cite{Ramos:SPHTBG}.

\begin{proposition}
\label{proposition: degree lower bound}
If $\graf$ is normal,
then $\growthdeg{\graf}{i}{\field}\ge \Ramos{\graf}{i}-1$.
\end{proposition}
\begin{proposition}
\label{proposition: degree upper bound}
If $\graf$ is normal, then $\growthdeg{\graf}{i}{\field}\le \Ramos{\graf}{i}-1$.
\end{proposition}

Assuming these results for the moment, we complete the proof.

\begin{proof}[Proof of Theorem \ref{thm:polynomial growth}]
We may assume that $\graf$ is smooth. Since the presence or absence of finitely many isolated vertices in $\graf$ does not change the eventual growth rate, we may further assume that each connected component of $\graf$ contains an edge. 
Thus, each connected component of $\graf$ is either normal, an interval, or a cycle.
The theorem is known for each of these components; indeed, Propositions \ref{proposition: degree lower bound} and \ref{proposition: degree upper bound} supply the normal case, and the other cases are classical.
This observation forms the base case for an induction on $|\pi_0(\graf)|$.

For the inductive step, we make use of ``big theta'' notation \cite[page 19]{Knuth:BOBOBT} defined as follows:
\[
f(k)=\Theta(g(k))\Longleftrightarrow
\exists c, C>0 : c\cdot g(k)\le f(k)\le C\cdot g(k) \quad\forall k\gg0.
\] Consider a decomposition of a smooth graph into two non-empty graphs $\graf=\graf_1\sqcup \graf_2$, and assume that the degree theorem holds for each of $\graf_1$ and $\graf_2$. 
By the K\"unneth theorem, we have
\begin{align*}
\dim_\field(H_i(B_k(\graf);\field))&=
\sum_{\ell=0}^k\sum_{a+b=i}
\dim_\field(H_a(B_\ell(\graf_1);\field))\dim_\field(H_b(B_{k-\ell}(\graf_2);\field))\\
&=\sum_{a+b=i}\Theta\left(\sum_{\ell=0}^k\ell^{\Ramos{\graf_1}{a}-1}(k-\ell)^{\Ramos{\graf_2}{b}-1}\right)\\
&=\sum_{a+b=i}\Theta\left(k^{\Ramos{\graf_1}{a}+\Ramos{\graf_2}{b}-1} \sum_{\ell=0}^k\frac 1k\left(\frac{\ell}k\right)^{\Ramos{\graf_1}{a}-1}\left(1-\frac\ell{k}\right)^{\Ramos{\graf_2}{b}-1}\right)\\
&=\sum_{a+b=i}\Theta\left(k^{\Ramos{\graf_1}{a}+\Ramos{\graf_2}{b}-1}\right) \int_0^1 x^{\Ramos{\graf_1}{a}-1}(1-x)^{\Ramos{\graf_1}{b}-1} dx,
\end{align*} where the second equality uses our assumption on $\graf_1$ and $\graf_2$. 
By our assumption on the connected components of $\graf$, both $\Ramos{\graf_1}{a}$ and $\Ramos{\graf_2}{b}$ are nonzero for every $a$ and $b$, so the integral shown is a nonnegative rational number independent of $k$.
The largest exponent of $k$ present in this sum is
\[
\max_{a+b=i}\left( \Ramos{\graf_1}{a}+\Ramos{\graf_2}{b}-1 \right)
=
\Ramos{\graf}{i}-1,
\] as desired---see Example~\ref{example:Ramos calculations}.
\end{proof}

\subsection{Lower bound}

To prove Proposition~\ref{proposition: degree lower bound}, we will define a special kind of homology class, called a $W$-torus, depending on a set of vertices $W$.
For a suitable choice of $W$, we will show that a $W$-torus exists and generates an $\field[E]$-submodule with polynomial growth of the expected degree. 

Throughout this section, we assume that $\graf$ is normal. Given a set $W$ of vertices, we write $\graf_W$ for the graph obtained from $\graf$ by successive explosion of each vertex of $W$.

\begin{definition}
A subset $W\subseteq \essential$ is \emph{well-separating} if the open star of each $v\in W$ intersects more than one connected component of $\graf_W$.
\end{definition}

\begin{figure}[ht]
\begin{center}
\begin{tikzpicture}[baseline=-0.5ex]
\begin{scope}[xshift=-3cm]
\draw (0,-1.2) ++ (150:0.8) -- ++ (0,2.4) -- ++ (150:1.2) -- ++ (0,-2.4) -- cycle;
\draw (0,.05) ++ (150:1.4) node[above] {\slantbox{0.866}{-0.5}{$\graf_{\!1}$}};
\draw (0, -0.2) ++ (30:0.8) -- ++ (0,1.4) -- ++ (30:1.2) -- ++ (0,-1.4) -- cycle;
\draw (0,0.3) ++ (30:1.5) node[above] {\slantbox{0.866}{0.5}{$\graf_{\!2}$}};
\draw[line width=3, white, fill=white] (0,-1.2) ++ (200:1.2) -- ++ (0,2.4) -- ++ (200:1.2) -- ++ (0,-2.4) -- cycle;
\draw[line width=3, white, fill=white] (0,-1.2) ++ (-20:1.2) -- ++ (0,2.4) -- ++ (-20:1.2) -- ++ (0,-2.4) -- cycle;
\draw[densely dashed] (0,-1.2) ++ (150:0.8) -- ++ (0,2.4) -- ++ (150:1.2) -- ++ (0,-2.4) -- cycle;
\draw[densely dashed] (0, -0.2) ++ (30:0.8) -- ++ (0,1.4) -- ++ (30:1.2) -- ++ (0,-1.4) -- cycle;
\draw[line width=3, white] (0,-1.2) ++ (200:1.2) -- ++ (0,2.4) -- ++ (200:1.2) -- ++ (0,-2.4) -- cycle;
\draw[line width=3, white] (0,-1.2) ++ (-20:1.2) -- ++ (0,2.4) -- ++ (-20:1.2) -- ++ (0,-2.4) -- cycle;
\draw[line width=3, white] (0,1) + (200:1.2) -- (0,1) -- +(-20:1.2);
\draw (0,1) + (150:.8) -- (0,1) -- +(30:.8);
\draw (0,0.7) + (30:.8) -- (0,1);
\draw (0,0) + (150:.8) -- (0,0) -- +(30:.8);
\draw (0,-0.3) + (150:.8) -- (0,0);
\draw (0,-1) + (150:.8) -- (0,-1);
\draw (0, -0.7) + (150:.8) -- (0,-1);
\draw (0,1) + (200:1.2) -- (0,1) -- +(-20:1.2);
\draw (0,1) + (200:1.2) -- (0,1) -- +(-20:1.2);
\draw (0,0) -- +(-20:1.2);
\draw (0,.3) + (-20:1.2) -- (0,0);
\draw (0,-0.7) + (-20:1.2) -- (0,-1);
\draw (0,-0.4) + (-20:1.2) -- (0,-1);
\draw (0,-1) + (200:1.2) -- (0,-1) -- +(-20:1.2);
\draw[fill] (0,1) circle (2.5pt);
\draw[fill] (0,0) circle (2.5pt);
\draw[fill] (0,-1) circle (2.5pt);
\draw (0,-1.2) ++ (200:1.2) -- ++ (0,2.4) -- ++ (200:1.2) -- ++ (0,-2.4) -- cycle;
\draw (0,0) ++ (200:1.7) node {\slantbox{0.939}{0.342}{$\graf_{\!4}$}};
\draw (0,-1.2) ++ (-20:1.2) -- ++ (0,2.4) -- ++ (-20:1.2) -- ++ (0,-2.4) -- cycle;
\draw (0,-.1) ++ (-20:1.75) node {\slantbox{0.939}{-0.342}{$\graf_{\!3}$}};
\draw[dashed] (-0.4, 1.5) -- (0.4, 1.5) -- (0.4, -1.5) -- (-0.4, -1.5) -- cycle;
\draw (0,1.5) node[above] {$W$};
\end{scope}
\begin{scope}[xshift=3cm]
\draw (-0.8,-1.4) -- ++ (0, 2.8) -- ++ (-1.2,0) -- ++ (0,-2.8) -- cycle;
\draw (-1.3,0) node {$\graf_{\!1}$};
\draw (0,-1.2) ++ (40:1) -- ++ (0,1.73) -- ++ (40:1) -- ++ (0,-1.73) -- cycle;
\draw (0,-0.4) ++ (40:1.6) node {\slantbox{0.766}{0.642}{$\graf_{\!2}$}};
\draw[red] (-0.8,1.2) -- (0,1);
\draw[red] (-0.8,1) -- (0,1);
\draw[red] (-0.8,0.8) -- (0,1);
\draw (-0.8, 0.23) -- (0,0.33) -- (-0.8,0.43);
\draw (-0.8,-1) -- (0,-1);
\draw[line width=3,white,fill] (0,-1.2) ++ (-20:1.2) -- ++ (0,1.73) -- ++ (-20:1) -- ++ (0,-1.73) -- cycle;
\draw[densely dashed] (0,-1.2) ++ (40:1) -- ++ (0,1.73) -- ++ (40:1) -- ++ (0,-1.73) -- cycle;
\draw[line width=3,white] (0,-1.2) ++ (-20:1.2) -- ++ (0,1.73) -- ++ (-20:1) -- ++ (0,-1.73) -- cycle;
\draw[line width=3,white] (0,-1) + (-20:1.2) -- (0,-1);
\draw (0,-1) + (40:1) -- (0,-1);
\draw[line width=3,white] (0,-0.33) + (-20:1.2) -- (0,-0.33);
\draw (0,-0.23) + (40:1) -- (0,-0.33);
\draw (0,-0.43) + (40:1) -- (0,-0.33);
\draw[line width=3,white] (0,0.33) + (-20:1.2) -- (0,0.33);
\draw (0,0.33) + (40:1) -- (0,0.33);
\draw (0,0.33) + (-20:1.2) -- (0,0.33);
\draw (0,-0.33) + (-20:1.2) -- (0,-0.33);
\draw (0,-1) + (-20:1.2) -- (0,-1);
\draw (0,-1.2) ++ (-20:1.2) -- ++ (0,1.73) -- ++ (-20:1) -- ++ (0,-1.73) -- cycle;
\draw (0,-0.4) ++ (-20:1.75) node {\slantbox{0.939}{-0.342}{$\graf_{\!3}$}};
\draw[fill,red] (0,1) circle (2.5pt);
\draw[fill] (0,0.33) circle (2.5pt);
\draw[fill] (0,-0.33) circle (2.5pt);
\draw[fill] (0,-1) circle (2.5pt);
\draw[dashed] (-0.4, 1.5) -- (0.4, 1.5) -- (0.4, -1.5) -- (-0.4, -1.5) -- cycle;
\draw (0,1.5) node[above] {$W'$};
\end{scope}
\end{tikzpicture}
\end{center}
\caption{A well-separating subset $W$ and a non-well-separating subset $W'$ (assuming $\graf_1\setminus W'$ is connected)}
\end{figure}
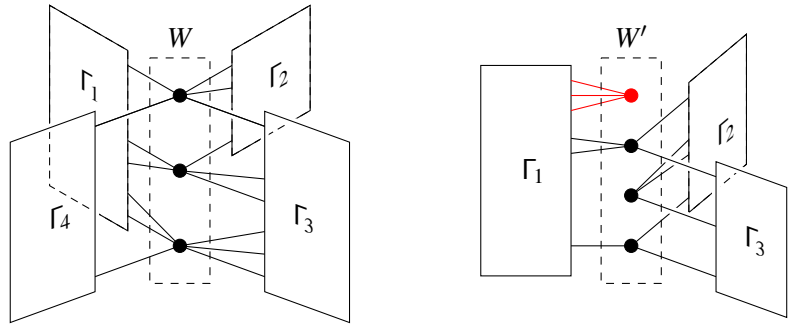

The empty set is vacuously well-separating.
We fix a field $\field$ and a subset $W\subseteq \essential$ for the remainder of this section, and
we write $F_{\ell}H_*(B(\graf);\field)$ for the filtration on homology induced by filtering $\fullyreduced{\graf;\field}$ by the number of half-edge generators at vertices of $W$. 

\begin{definition}
\label{defi: W-torus}
A class $\alpha\in H_{|W|}(B(\graf);\field)$ is a \emph{$W$-torus} if $\alpha$ is the external product of classes $\{\alpha_v\}_{v\in W}$, where $\alpha_v$ is a star class at $v$.
We call $\alpha_v$ a \emph{star factor} of $\alpha$.
We further say that $\alpha$ is \emph{rigid} if $W$ is well-separating and $\alpha$ lies in $F_{|W|}H_{|W|}(B(\graf);\field)$ but not in $F_{|W|-1}H_{|W|}(B(\graf);\field).$
\end{definition}
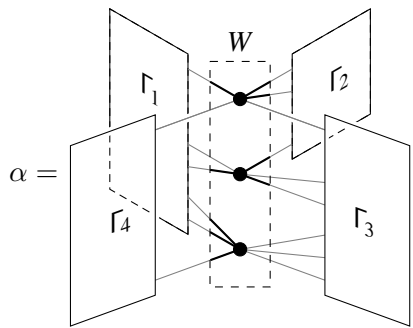
\begin{figure}[ht]
\[
\alpha=
\begin{tikzpicture}[baseline=-0.5ex]
\begin{scope}
\draw (0,-1.2) ++ (150:0.8) -- ++ (0,2.4) -- ++ (150:1.2) -- ++ (0,-2.4) -- cycle;
\draw (0,.05) ++ (150:1.4) node[above] {\slantbox{0.866}{-0.5}{$\graf_{\!1}$}};
\draw (0, -0.2) ++ (30:0.8) -- ++ (0,1.4) -- ++ (30:1.2) -- ++ (0,-1.4) -- cycle;
\draw (0,0.3) ++ (30:1.5) node[above] {\slantbox{0.866}{0.5}{$\graf_{\!2}$}};
\draw[line width=3, white,fill] (0,-1.2) ++ (200:1.2) -- ++ (0,2.4) -- ++ (200:1.2) -- ++ (0,-2.4) -- cycle;
\draw[densely dashed] (0,-1.2) ++ (150:0.8) -- ++ (0,2.4) -- ++ (150:1.2) -- ++ (0,-2.4) -- cycle;
\draw[line width=3, white] (0,-1.2) ++ (200:1.2) -- ++ (0,2.4) -- ++ (200:1.2) -- ++ (0,-2.4) -- cycle;
\draw[line width=3, white,fill] (0,-1.2) ++ (-20:1.2) -- ++ (0,2.4) -- ++ (-20:1.2) -- ++ (0,-2.4) -- cycle;
\draw[densely dashed] (0, -0.2) ++ (30:0.8) -- ++ (0,1.4) -- ++ (30:1.2) -- ++ (0,-1.4) -- cycle;
\draw[line width=3, white] (0,-1.2) ++ (-20:1.2) -- ++ (0,2.4) -- ++ (-20:1.2) -- ++ (0,-2.4) -- cycle;
\draw[line width=3, white] (0,1) + (200:1.2) -- (0,1) -- +(-20:1.2);
\draw[gray] (0,1) + (150:.8) -- (0,1) -- +(30:.8);
\draw[gray] (0,0.7) + (30:.8) -- (0,1);
\draw[gray] (0,0) + (150:.8) -- (0,0) -- +(30:.8);
\draw[gray] (0,-0.3) + (150:.8) -- (0,0);
\draw[gray] (0,-1) + (150:.8) -- (0,-1);
\draw[gray] (0, -0.7) + (150:.8) -- (0,-1);
\draw[gray] (0,1) + (200:1.2) -- (0,1) -- +(-20:1.2);
\draw[gray] (0,1) + (200:1.2) -- (0,1) -- +(-20:1.2);
\draw[gray] (0,0) -- +(-20:1.2);
\draw[gray] (0,.3) + (-20:1.2) -- (0,0);
\draw[gray] (0,-0.7) + (-20:1.2) -- (0,-1);
\draw[gray] (0,-0.4) + (-20:1.2) -- (0,-1);
\draw[gray] (0,-1) + (200:1.2) -- (0,-1) -- +(-20:1.2);
\draw[fill] (0,1) circle (2.5pt);
\draw[fill] (0,0) circle (2.5pt);
\draw[fill] (0,-1) circle (2.5pt);
\draw (0,-1.2) ++ (200:1.2) -- ++ (0,2.4) -- ++ (200:1.2) -- ++ (0,-2.4) -- cycle;
\draw (0,0) ++ (200:1.7) node {\slantbox{0.939}{0.342}{$\graf_{\!4}$}};
\draw (0,-1.2) ++ (-20:1.2) -- ++ (0,2.4) -- ++ (-20:1.2) -- ++ (0,-2.4) -- cycle;
\draw (0,-.1) ++ (-20:1.75) node {\slantbox{0.939}{-0.342}{$\graf_{\!3}$}};
\draw[dashed] (-0.4, 1.5) -- (0.4, 1.5) -- (0.4, -1.5) -- (-0.4, -1.5) -- cycle;
\draw (0,1.5) node[above] {$W$};
\draw[thick] (0,1) + (150:.46) -- (0,1) -- +(30:.46);
\draw[thick] (0,0.8275) + (30:.46) -- (0,1);
\draw[thick] (0,0) -- +(30:.46);
\draw[thick] (0,-0.1725) + (150:.46) -- (0,0);
\draw[thick] (0,0) -- +(-20:0.42);
\draw[thick] (0,-1) -- + (150:.46);
\draw[thick] (0, -0.8275) + (150:.46) -- (0,-1);
\draw[thick] (0,-1) + (200:0.42) -- (0,-1);
\end{scope}
\end{tikzpicture}
\]
\caption{A $W$-torus which Lemma~\ref{lemma:torus module} shows to be rigid}
\end{figure}

\begin{example}
There is a unique $\varnothing$-torus, namely the class of the empty configuration in $H_0(B_0(\graf);\field)$, and this class is rigid since $F_{-1}H_0(B(\graf);\field)=0$.
\end{example}

\begin{example}\label{example:non-rigid star class}
Consider the theta graph $\thetagraph{3}$, which is the double cone on three points with vertices the cone points $v$ and $w$ as depicted in Figure~\ref{figure:theta graph}. According to the $\Theta$-relation~\cite[Lemma 5.8(4) and Definition 5.9(4)]{AnDrummond-ColeKnudsen:SSGBG}, any star class $\alpha_v$ at $v$ is equal to some star class $\alpha_w$ at $w$. 
We have defined rigidity only for well-separating $W$ because if $W$ is not well-separating in $\graf$, pulling the $\Theta$-relation back along a smoothing $\thetagraph{3}'\to \thetagraph{3}$ and then pushing forward along a graph embedding $\thetagraph{3}'\to \graf$ shows that no $W$-torus could ever have a rigidity property.

We can use the same strategy to give examples of non-rigid $W$-tori in cases where $W$ is well-separating. 
To wit, consider the graph $\thetagraph{3}^+$ of Figure~\ref{figure:theta graph}.
There is an evident graph morphism $\iota$ from $\thetagraph{3}$ to $\thetagraph{3}+$.
The vertex set $W=\{v'\}$ is well-separating in $\thetagraph{3}^+$ but the star class $\iota(\alpha_v)$ is not rigid because it is equal to the star class $\iota(\alpha_w)$ by pushing forward the $\Theta$-relation.
\end{example}
\begin{figure}[ht]
\begin{center}
\begin{tikzpicture}
\begin{scope}
\fill[black] (-.5,0) circle (2.5pt) node[left] {$v\ $};
\fill[black] (.5,0) circle (2.5pt) node[right] {$\ w$};
\draw(0,0) circle (.5cm);
\draw(.5,0)--(-.5,0);
\draw(0,-.6) node[below]{$\ \ \thetagraph{3}$};
\end{scope}
\begin{scope}[xshift=4cm]
\fill[black] (-.5,0) circle (2.5pt) node[above] {$v'\ \,$};
\fill[black] (-1.2,0) circle (2.5pt);
\fill[black] (.5,0) circle (2.5pt) node[right] {$w'$};
\draw(0,0) circle (.5cm);
\draw[line width=1.5] (-.25,.433) arc [start angle=120, end angle=240, radius=.5cm];
\draw[line width=1.5] (-.5,0) -- (-.25,0);
\draw(.5,0)--(-1.2,0);
\draw(0,-.6) node[below]{$\ \ \thetagraph{3}^+$};
\end{scope}
\end{tikzpicture}
\end{center}
\caption{The theta graph $\thetagraph{3}$ and a graph $\thetagraph{3}^+$ with a well-separating vertex set $W=\{v'\}$ and a non-rigid $W$-torus (indicated in bold)}
\label{figure:theta graph}
\end{figure}
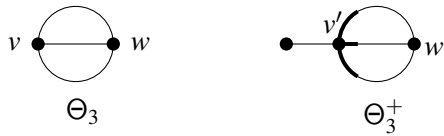

\begin{remark}[Universal Generators]
\label{remark: Universal Generators}
The definition and use of $W$-tori can be viewed as an outgrowth of the study of an important question in configuration spaces of graphs.

This question is to determine a \emph{universal presentation} for the homology of configuration spaces of graphs (or of specific classes of graphs). 
By a universal presentation we mean 
\begin{enumerate}
\item a set of homology classes (\emph{universal generators}) which generate all homology groups of configuration spaces of graphs under pushforward along graph morphisms, edge stabilization, and external products
\item a set of \emph{universal relations} which generate all relations among universal generators under the same three operations.
\end{enumerate}

For example, the empty configuration is a universal generator for degree $0$ homology while star classes and \emph{loop classes} (see \cite[Definition 5.3]{AnDrummond-ColeKnudsen:SSGBG}) are universal generators for degree $1$ homology. 
The Proof of Theorem~\ref{thm:polynomial growth} for the special case of trees in~\cite[Theorem V.3]{MaciazekSawicki:HGP1CG} amounts to showing that star classes are in fact universal generators for all homology groups for trees (it is known that this is not true for general graphs). 
In later work, the same authors show that the same statement holds for wheel graphs~\cite[\textsection 5.4]{MaciazekSawicki:NAQSG}.

One small step in the study of universal presentations is to ask quantitatively how much of the homology for a given graph or class of graphs is generated by such tori in general. 
Our proof of Proposition~\ref{proposition: degree lower bound} implies that for any graph and in any homological degree the amount of homology generated by such tori makes up a positive proportion of all homology as the weight increases. 
But this is a very crude estimate and it would be useful to have further information about this in order to know where to look for new universal generators and relations.
\end{remark}

Set $R_W=\field[E_W]$ and write $\pi_W$ for the projection from $\field[E]$ to $R_W$. If $\alpha$ is a $W$-torus, then the surjection $\field[E]\langle \alpha\rangle\to\field[E]\cdot\alpha$ factors through $\pi_W$, since any two edges in a connected component of $\graf_W$ may be connected by a path of edges disjoint from the standard representative of $\alpha$.

\begin{lemma}
\label{lemma:torus module}
Let $W$ be well-separating and let $\alpha$ be a $W$-torus. The following are equivalent.
\begin{enumerate}
\item \label{item: torus rigid} The $W$-torus $\alpha$ is rigid.
\item \label{item: expected module structure} The action of $\field[E]$ induces an isomorphism \[R_W\langle \alpha\rangle\xrightarrow{\simeq}\field[E]\cdot\alpha.\]
\item \label{item: star factor condition} The standard representative of every star factor $\alpha_v$ involves edges in at least two distinct connected components of $\graf_W$. 
\end{enumerate}
\end{lemma} 
The third condition is a little technical but is also easier than the other two to check. 
In some sense it says that to guarantee rigidity, it is enough to avoid the class of ``obvious'' pitfalls illustrated in Example~\ref{example:non-rigid star class}.
In particular, it has the following corollary.
\begin{corollary}
\label{cor: existence of rigid tori}
For any well-separating subset $W$, there is a rigid $W$-torus.
\end{corollary}
\begin{proof}
Choose a star class $\alpha_v$ at each $v\in W$. Since $W$ is well-separating, the star class $\alpha_v$ may be chosen to satisfy the star factor condition~\eqref{item: star factor condition}. 
\end{proof}
\begin{remark}
For an arbitrary set of vertices $W$, one could instead filter according to a maximal well-separating subset $W_0\subseteq W$. The resulting generalized rigid $W$-tori do not exist for arbitrary $W$, $\field$, and $\graf$; however, if $\field=\mathbb{Z}/2\mathbb{Z}$ or if $\graf$ is planar, then such classes exist for any $W$. 
To construct one, use Corollary~\ref{cor: existence of rigid tori} to generate a rigid $W_0$-torus and form the external product with a star class at each vertex in $W\setminus W_0$.
Either assumption implies that this external product is nonzero, and the proof of rigidity proceeds along similar lines.
\end{remark}

In the situation of interest, we can guarantee that $W$ is well-separating.

\begin{lemma}
\label{lemma: maximal W are well-separating}
If $W\subseteq \essential$ with $|W|=i$ is such that $
\Ramos{\graf}{W} = \Ramos{\graf}{i},$
then either $W$ is well-separating or $i=\Ramos{\graf}{i}=1$.
\end{lemma}

Assuming these lemmas, we can establish the lower bound.

\begin{proof}[Proof of Proposition~\ref{proposition: degree lower bound}]

Let $W\subseteq \essential$ be a subset with $|W|=i$ such that $\Ramos{\graf}{W}=\Ramos{\graf}{i}$. 
Except in the special case $i=\Ramos{\graf}{i}=1$, Lemma~\ref{lemma: maximal W are well-separating} guarantees that $W$ is well-separating, so Corollary~\ref{cor: existence of rigid tori} supplies a rigid $W$-torus $\alpha$. 
The dimension of $H_i(B_k(\graf);\field)$ is no less than that of the submodule $\field[E]\cdot\alpha$, which is isomorphic to $R_W\langle \alpha\rangle$ by Lemma~\ref{lemma:torus module}.
Suppose $\alpha$ is in $H_i(B_{k_0}(\graf);\field)$.
Then the $\field$-dimension of the space of monomials in $R_W$ of weight $k-k_0$ is exactly $\binom{\Ramos{\graf}{W}+(k-k_0)-1}{\Ramos{\graf}{W}-1}$ and hence $\field[E]\cdot\alpha$ has polynomial growth of degree $\Ramos{\graf}{W}-1$.

For the special case $i=\Ramos{\graf}{i}=1$, we must verify that $H_1(B_k(\graf);\field)$ is eventually non-zero, which is well known---see \cite[Theorem 3.16]{KoPark:CGBG}, for example.
\end{proof}

It remains to prove Lemmas~\ref{lemma:torus module} and \ref{lemma: maximal W are well-separating}. In the proof, we make use of the spectral sequence arising from the filtration $F_{\ell}H_*(B(\graf);\field)$ introduced above, which we denote by $E^r_{p,q}$. This spectral sequence is a spectral sequence of $\field[E]$-modules.

\begin{proof}[Proof of Lemma~\ref{lemma:torus module}]

The unique $\varnothing$-torus is rigid, and clearly $\field[E]\cdot[\varnothing]\cong R_\varnothing$. In the remainder of the proof, we take $W\neq\varnothing$. 

We prove first that~\eqref{item: torus rigid} implies~\eqref{item: expected module structure}.
Suppose that $\alpha$ is rigid, and write \[[\alpha]=\alpha + F_{|W|-1}H_{|W|}(B(\graf);\field)\in E^\infty_{|W|,0}\] for the resulting element in the associated graded. 
Since $\alpha$ is rigid, $[\alpha]$ cannot be zero.
We wish to show that the lefthand map in the composite \[R_W\langle\alpha\rangle\to \field[E]\cdot\alpha\to \field[E]\cdot [\alpha]\] is an isomorphism. Since this map is surjective, it suffices to show that the composite is injective. We will show that $E^\infty_{|W|,0}$ is a submodule of a free $R_W$-module $R_W\langle X\rangle$ for some set $X$. Writing $\{[\alpha]_x\}_{x\in X}$ for the coordinates of $[\alpha]$ in this basis, the kernel of the map in question is $\bigcap_{x\in X}\mathrm{Ann}_{R_W}([\alpha]_x)$. Since $[\alpha]\neq0$ and $R_W$ has no zero divisors, one of the terms of this intersection is zero, implying the claim.

Our spectral sequence is concentrated in the first quadrant, so we have the containment $E^\infty_{|W|,0}\subseteq E^2_{|W|,0}$, and, since $E^1_{|W|+1,0}=0$, we have the further containment $E^2_{|W|,0}\subseteq E^1_{|W|,0}$. Both are containments of $\field[E]$-submodules. Since $\graf$ has no isolated vertices, there is a $\field[E]$-linear isomorphism \[E^1_{|W|,0}\cong H_0(B(\graf_W);\field)\otimes \field\langle X\rangle\cong R_W\otimes \field\langle X\rangle,\] where $X$ is the set of generators of $\fullyreduced{\graf}$ of the form $\bigotimes_{v\in W} h^v$ with $h^v$ a difference of half-edge generators
 at $v$. This module is the desired $R_W$-free module.

We prove that~\eqref{item: expected module structure} implies~\eqref{item: star factor condition} by proving the contrapositive.
Let $\alpha_v$ be a star factor represented by a graph morphism $\iota:\stargraph{3}\to \graf$ whose image intersects only one connected component of $\graf_W$, and extend $\iota$ to a graph morphism $\widehat{\iota}:\thetagraph{3}'\to \graf$ with the same property, where $\thetagraph{3}'$ is a subdivision of $\thetagraph{3}$ (in other words, there is a smoothing $f:\thetagraph{3}'\to\thetagraph{3}$).
By pulling back via $f$ the $\Theta$-relation on $\thetagraph{3}$ as seen in Example~\ref{example:non-rigid star class}, we have the $\Theta$-relation on $\thetagraph{3}'$ as well and so $\alpha_v=\alpha_w$ for some star class $\alpha_w$ at $w$ if we write $w$ for the image of the other essential vertex of $\thetagraph{3}'$ under $\widehat\iota$.
Replacing $\alpha_v$ by $\alpha_w$ in the external product defining $\alpha$, we conclude that $\alpha$ has a representative $a$ that does not involve $v$. Since $W$ is well-separating, there are edges $e_i$ and $e_j$ at $v$ lying in distinct connected components of $\graf_W$, and appending the generator $h_{ij}$ to $a$ produces an element whose boundary is $(e_i-e_j)\alpha$. Thus, the map $R_W\langle \alpha\rangle\to \field[E]\cdot\alpha$ has nonzero kernel.

Finally, we prove that~\eqref{item: star factor condition} implies~\eqref{item: torus rigid}. Throughout, products and tensor products are indexed by $W$.
At each $v\in W$, choose a basis for the half-edge generators at $v$ such that $\alpha_v$ is represented by $h^v_{12}(e^v_1-e^v_3)-h^v_{13}(e^v_1-e^v_2)$.
By our assumption~\eqref{item: star factor condition}, we may assume that $e^v_1$ and $e^v_3$ intersect distinct connected components of $\graf_W$. We claim that $\bigotimes h^v_{12}$ appears with a nonzero $\field[E]$-coefficient in every chain representative of $\alpha$, implying in particular that $\alpha\notin F_{|W|-1}H_{|W|}(B(\graf);\field)$, as desired.

We begin by examining the coefficient $\prod (e^v_1-e^v_3)$ of $\bigotimes h^v_{12}$ in the standard representative $a$ of $\alpha$. 
We apply $\pi_W$ to this coefficient to obtain an element of $R_W$. 
By construction, $\pi_W(e^v_1-e^v_3)\neq0$ for each $v$.
Choose an ordering of $\pi_0(\graf_W)$ and write $\pi_W(e^v_1-e^v_3)=\pm (\graphfont{\Delta}^v_i-\graphfont{\Delta}^v_j)$ with $\graphfont{\Delta}^v_i<\graphfont{\Delta}^v_j$ in the ordering.
The leading term $\prod \graphfont{\Delta}^v_i$ appears with coefficient $\pm 1$ in the polynomial $\pi_W(\prod e^v_1-e^v_3)$, so $\prod (e^v_1-e^v_3)\neq 0$, establishing the claim for the standard representative $a$.

Suppose now that $a-a'=\partial c$. In our preferred basis for the tensor product, $c$ is necessarily a sum of elementary tensors of $(|W|+1)$ half-edge generators, some of which are of the form
\begin{equation}
\label{eq: shape of maximal elementary tensor}
p(E)h_{ij}\otimes  \bigotimes h^v_{12}
\end{equation}
for some half-edge generator $h_{ij}$ at a vertex away from $W$.
Terms not of this form contain strictly fewer than $|W|$ tensor factors of the form $h^v_{12}$, and the boundary of such a term is a sum of terms containing strictly fewer than $|W|$ tensor factors of the form $h^v_{12}$.
On the other hand, the boundary of~\eqref{eq: shape of maximal elementary tensor} is a sum of terms containing strictly fewer than $|W|$ tensor factors of the form $h^v_{12}$, together with the term
\[
p(E)(e_i-e_j)\bigotimes h^v_{12}.
\]
By our assumption on $h_{ij}$, $e_i$ and $e_j$ lie in the same connected component of $\graf_W$, so $\pi_W(p(E)(e_i-e_j))=0$.
Thus, the coefficient of $\bigotimes h^v_{12}$ in $a'$ coincides with the coefficient in $a$ modulo $\ker\pi_W$, so this coefficient is nonzero by the previous paragraph.
\end{proof}

\begin{remark}
The same techniques serve to establish a version of Lemma \ref{lemma:torus module} for toric classes with some factors given by loop classes. 
\end{remark}

\begin{proof}[Proof of Lemma \ref{lemma: maximal W are well-separating}]
If $|W|=0$, or if $|W|=1$ and $\Ramos{\graf}{W}>1$, then the claim is obvious. If $|W|>1$ and $W$ is not well-separating, then there is a vertex $v\in W$ whose open star intersects only one connected component of $\graf_W$. We claim that there is a second vertex $v'\in \essential\setminus W$ lying in this component of $\graf_W$ and sharing an edge $e$ with some vertex of $W\setminus\{v\}$. 

Given such a vertex, we replace $v$ with $v'$ in $W$ to obtain a set $W'$ with $\Ramos{\graf}{W'}\ge \Ramos{\graf}{W}+1$. Indeed, the relations ${\sim}_W$ and ${\sim}_{W\setminus\{v\}}$ are identical, but the equivalence class of $e$ under the relation ${\sim}_{W}$ is split into at least two distinct equivalence classes under ${\sim}_{W'}$. Thus, $\Ramos{\graf}{W}$ does not achieve the maximum value $\Ramos{\graf}{i}$.

Assume for contradiction that such a $v'$ does not exist, and suppose there is a reduced edge path in $\graf$ beginning at $v$, terminating at some other vertex of $W$, and passing through no third vertex of $W$. Such a path is necessarily contained in the connected component of $\graf_W$ containing the open star of $v$. Since this path is reduced and $\graf$ has no bivalent vertices, the path passes only through essential vertices. We obtain a contradiction of the non-existence of $v'$ unless all such paths have length one; in other words, every edge at $v$ terminates either at a vertex of $W$ or a vertex of valence $1$. Therefore, either the valence of $v$ is $1$, or the open star of $v$ intersects more than one component of $\graf_W$. Either is a contradiction, so we conclude that no edge path exists between $v$ and any other vertex of $W$. Since $|W|>1$, it follows that $\graf$ is disconnected, a contradiction.
\end{proof}

\subsection{Upper bound}

Let $\graf$ be a graph and $\field$ be a field. Given a subset $W\subseteq\essential$ and a vertex $v$, we write $E_W(v)\subseteq E_W$ for the set of equivalence classes of edges edges adjacent to $v\in V$. We set \[\fullyreducedsub{W}{\graf;\field}\coloneqq \fullyreduced{\graf_W;\field}\otimes_{\field[E]}\field[E_W].\] The action of $\field[E]$ on $\field[E_W]$ is weight-respecting so the complex $\fullyreducedsubnoarg{W}$ is weight-graded. Since $\fullyreducedsubnoarg{\varnothing}$ is simply the reduced \'Swi\k{a}tkowski complex over $\field$, Proposition~\ref{proposition: degree upper bound} is a special case of the following result.

\begin{proposition}\label{proposition: degree count with sinks}
Let $\graf$ be smooth and contain no isolated vertices. 
For any $W\subseteq \essential$, the dimension of $H_i(\fullyreducedsub{W}{\graf;\field})$ is eventually polynomial in the weight of degree at most $\Delta_\graf^i-1$.
\end{proposition}

\begin{remark}
Geometrically, the weight $k$ subcomplex of $\fullyreducedsubnoarg{W}(\graf)$ corresponds to the space of configurations of $k$ points in $\graf$ which are permitted to collide at vertices in $W$---see \cite[\textsection 2]{ChettihLuetgehetmann:HCSTL} and \cite[\textsection 1.1]{Ramos:CSGCTPC}.

Note in particular that the behavior of $\fullyreducedsub{W}{\graf}$ is different from that of $\fullyreduced{\graf_W}$. 
In both cases, there is no local homological information near $W$, but particles pass freely through vertices in $W$ in the former while in the latter they avoid $W$. 
This is easiest to see in the star graph $\stargraph{n}$, where the former models a contractible space and the latter a space with $n$ contractible components.
\end{remark}

The strategy of the proof of Proposition \ref{proposition: degree count with sinks} will be to show that the desired growth rate is already achieved on the $E^2$ page of a certain spectral sequence converging to the homology of $\fullyreducedsub{W}{\graf;\field}$. In order to introduce this spectral sequence, we require the following notation. Given a vertex $v\in V\setminus W,$ we write $\stargraph{v}$ for a star graph of maximal valence equipped with a graph morphism $\iota:\stargraph{v}\to\graf$ sending the star vertex to $v$. Setting $\Wvring{}=\field[E_W(v)]$, we write $\Wvcomplex$ for the chain complex 
\[
\Wvcomplex\coloneqq
\fullyreduced{\stargraph{v};\field}\otimes_{\field[E(\stargraph{v})]}\Wvring
\]
of $\Wvring$-modules. 
The chain complex $\Wvcomplex$ is finitely generated over $\Wvring$ because $\fullyreduced{\stargraph{v};\field}$ is finitely generated over $\field[E(\stargraph{v})]$.

\begin{lemma}\label{lem:ss of bicomplex}
Let $v$ be an essential vertex of $V\setminus W$ in $\graf$.
There is a convergent, homological, weight-graded spectral sequence of $\field[E_W]$-modules

\[E^2_{p,q}\cong 
H_{p}\left(\fullyreducedsub{W}{\graf_v;\field}\otimes_{\Wvring} H_q(\Wvcomplex)\right)\implies H_{p+q}(\fullyreducedsub{W}{\graf;\field}).\]
\end{lemma}
\begin{proof}
We write the differential of $\fullyreducedsub{W}{\graf; \field}$ as $\partial = \partial_v + \partial_-$, where $\partial_v$ is the sum of the terms of the differential involving half-edges at $v$.
These two operators square to zero individually and commute, giving $\fullyreducedsub{W}{\graf; \field}$ the structure of a bicomplex of $\field[E_W]$-modules. The desired spectral sequence is the spectral sequence of this bicomplex with zeroth differential $\partial_v$ and first differential $\partial_-$---see Figure~\ref{figure: two row spectral sequence}. Since this bicomplex is concentrated in finitely many bidegrees, the spectral sequence collapses and in particular converges.

Using the decomposition $\fullyreduced{\graf}\cong \fullyreduced{\graf_v}\otimes_{\mathbb{Z}[E(v)]} \fullyreduced{\stargraph{v}}$, we obtain the isomorphism 
\[
E^0\cong \left(\fullyreducedsub{W}{\graf_v;\field}\otimes_{\Wvring} \Wvcomplex,\, 1\otimes\partial\right)
\] of trigraded $\field[E_W]$-modules, where the homological bidegree on the righthand side is the natural bigrading of the tensor product. Since $\fullyreducedsub{W}{\graf_v;\field}$ is $\field[E_W]$-free, and hence $\Wvring$-free, the K\"{u}nneth isomorphism gives \[
E^1\cong \left(\fullyreducedsub{W}{\graf_v;\field}\otimes_{\Wvring} H_*(\Wvcomplex),\, \partial\otimes1\right),
\] completing the proof.
\end{proof}

\begin{figure}[ht]
\[\begin{tikzcd}
E^0_{0,1}\dar{d^0}&\cdots\lar[dashed]{d^1} & E^0_{p-1,1}\lar[dashed]{d^1}\dar{d^0} & E^0_{p,1}\dar{d^0}\lar[dashed]{d^1} 
& \cdots\lar[dashed]{d^1}
&E^0_{|V|,1}\lar[dashed]{d^1}\dar{d^0}
\\
E^0_{0,0}&\lar[dashed]{d^1}\cdots & E^0_{p-1,0}\lar[dashed]{d^1} & E^0_{p,0}\lar[dashed]{d^1} 
& \cdots\lar[dashed]{d^1} 
&E^0_{|V|,0}\lar[dashed]{d^1}
\end{tikzcd}\]
\caption{The potentially nonzero entries of the $E^0$ page of the spectral sequence for the bicomplex $(\fullyreducedsub{W}{\graf;\field}, \partial_v,\partial_-)$}
\label{figure: two row spectral sequence}
\end{figure}
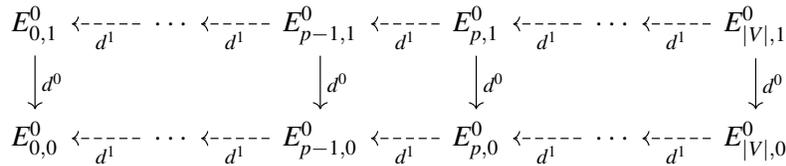

The key to the desired growth estimate is the following technical lemma.

\begin{lemma}\label{lem:technical growth}
Let $(M,\partial_M)$ be a differential bigraded $R_{W,v}$-module and $N$ a bigraded $R_{W,v}$-module with $N$ finitely generated and concentrated in strictly positive homological degrees. If $\dim_\field H_j(M)$ is eventually polynomial in weight of degree at most $d$ for every $0\leq j<i$, then $\dim_\field H_i(M\otimes_{R_{W,v}} N,\partial_M\otimes 1)$ is as well.
\end{lemma}
We assume this result for the moment.

\begin{proof}[Proof of Proposition \ref{proposition: degree count with sinks}]
Let $\graf$ be smooth with no isolated vertices. 
We proceed by induction on $|\essential\setminus W|$. 

For the base case $W=\essential$, the complex $\fullyreducedsub{W}{\graf;\field}$ is isomorphic to $\field[E_W]$ concentrated in degree 0. 
In weight $k$, the dimension of this vector space is the number of ways of putting $k$ indistinguishable balls in $E_W$ distinct bins, which is polynomial of degree $E_W-1=\Ramos{\graf}{|W|}-1$, as desired.

For the inductive step, suppose the statement has been shown true for $|\essential\setminus W|<r$ and let $|\essential\setminus W|=r>0$. 
Fixing a vertex $v$ in $\essential\setminus W$, it will suffice to show that the $E^2_{i,0}$ and $E^2_{i-1,1}$ entries of the spectral sequence of Lemma \ref{lem:ss of bicomplex} each have eventual polynomial growth of degree at most $\Delta_{\graf}^i-1$ (since $\Wvcomplex$ is concentrated in degrees 0 and 1, these are the only nonzero entries in the appropriate degree on the $E^2$ page).

We note first that $H_0(\Wvcomplex)$ is one-dimensional in each weight, since there is a degree 1 chain interpolating between any two edges of $\stargraph{v}$. Thus, \begin{align*}E^2_{i,0}&\cong H_i(\fullyreducedsub{W}{\graf_v;\field}\otimes_{R_{W,v}}\field)\\
&\cong H_i(\fullyreduced{\graf_{W\sqcup\{v\}};\field}\otimes_{\field[E]}\field[E_W]\otimes_{R_{W,v}}\field)\\
&\cong H_i(\fullyreducedsub{W\sqcup \{v\}}{\graf;\field}),
\end{align*}
and the latter has polynomial growth of degree at most $\Delta_\graf^i-1$ by induction (we have used the isomorphism $\field[E_{W\sqcup\{v\}}]\cong \field[E_W]\otimes_{R_{W,v}}\field$).

Next, we have $E^2_{i-1,1}\cong H_{i}(M\otimes_{R_{W,v}} N)$, where $M=\fullyreducedsub{W}{\graf_v;\field}$ and $N=H_1(\Wvcomplex)$. 
As previously noted, $\Wvcomplex$ is finitely generated over the Noetherian ring $R_{W,v}$, so $N$ is finitely generated, as well as concentrated in strictly positive degrees by definition. 
Moreover, $M$ is degreewise free and thus degreewise flat.
Since $\graf$ was smooth and contained no isolated vertices, the same is true of $\graf_v$.
Then the inductive hypothesis guarantees that, for $j<i$, the homology group $H_j(M)$ is eventually polynomial in the weight of degree at most \[\Delta_{\graf_v}^j-1\le \Delta_{\graf}^{j+1}-1\le \Delta_{\graf}^i-1,\]
so Lemma \ref{lem:technical growth} implies that $\dim_\field E^2_{i-1,1}$ has the same property.
\end{proof}

We conclude with the proof of the technical lemma.

\begin{proof}[Proof of Lemma \ref{lem:technical growth}]
Up to associated graded, the homology group $H_i(M\otimes_{R_{W,v}}N)$ is the sum of the $i$th anti-diagonal on the $E^\infty$ page of a first-quadrant (hence convergent) K\"unneth spectral sequence.
Thus, it suffices to bound the growth of the $i$th antidiagonal of \[E^2\cong \textstyle\Tor_{\Wvring}(H_*(M),N).\] 
To calculate these Tor groups, we resolve $N$. 
Since $\Wvring$ is Noetherian of finite global dimension, and since $N$ is finitely generated, there is a finite length resolution $P$ of $N$ by finitely generated projective $\Wvring$-modules. By the Quillen--Suslin theorem, $P$ is in fact a free resolution, generated by a trigraded $\field$-vector space $G=\bigoplus G_{a,b}\{k\}$ in which every summand is finite dimensional and almost all vanish.
The trigraded complex $H_*(M)\otimes_{\Wvring} P$ computes the desired (weight-graded) $\Tor$ groups. Since $N$ is concentrated in strictly positive homological degrees, $P$ is as well, so the $i$th antidiagonal of this complex is isomorphic to \[\bigoplus_{j=0}^{i-1}\bigoplus_{a+b=i-j}\bigoplus_kH_j(M)\otimes_\field G_{a,b}\{k\},\] which is a finite sum of vector spaces having eventual polynomial growth of degree at most $d$.
\end{proof}

\section{Edge formality}\label{section:formality}

We now undertake the in-depth study of one aspect of the chain level $R[E]$-module structure induced by edge stabilization, namely the question of its formality. 

\subsection{Edge formality} Recall that a differential bigraded $A$-module is said to be \emph{formal} if it is connected to its homology by a finite zig-zag of (bigraded) quasi-isomorphisms:
\[
(M,d)\xleftarrow{\simeq} (M_1,d_1)\xrightarrow{\simeq} (M_2,d_2)\xleftarrow{\simeq}\cdots \xrightarrow{\simeq} (H(M,d),0). 
 \] 

\begin{definition}
We say that a graph $\graf$ is \emph{edge formal} over the commutative ring $R$ if the the singular chain complex of $B(\graf)$ with coefficients in $R$ is formal as a differential bigraded $R[E]$-module.
\end{definition}

The goal of this section is to prove Theorem \ref{thm:formality}, which gives a complete characterization of edge formal graphs independent of coefficient ring. 
We first recall the statement of this theorem---see Definition \ref{def:small graph} for a reminder on terminology.

\begin{formality}\label{thm:formality2}
A graph is edge formal if and only if it is small.
\end{formality}

In proving this theorem, we may replace the large and unwieldy complex of singular chains with the smaller complex $S(\graf;R)$; indeed, as we will see below in Corollary \ref{cor:chain level comparison}, this $R[E]$-module is quasi-isomorphic to the $R[E]$-module of singular chains with coefficients in $R$. Moreover, we may work interchangeably with the complexes $S(\graf;R)$ and $\fullyreduced{\graf;R}$. Since the underlying bigraded $R[E]$-module of $S(\graf;R)$ (resp. $\fullyreduced{\graf;R}$) is free, we may use the following criterion, whose proof is a standard argument in homotopical algebra.

\begin{lemma}\label{lemma:formality}
Let $A$ be a weight-graded commutative ring and $M$ a differential bigraded $A$-module which is projective in each homological degree. Then $M$ is formal if and only if there is a map of differential bigraded $A$-modules $M\to H_*(M)$ inducing the identity on homology. 
\end{lemma}

Thus, Theorem \ref{thm:formality} amounts to characterizing the (non-)existence of such maps.

\begin{remark}
We caution the reader that the notion of an edge formal graph has no direct connection with that of a formal topological space in the sense of rational homotopy theory. 
For partial results about this other kind of formality, see~\cite[Theorems~1.2 and~1.3]{KoLaPark:GBGM}, where a necessary and sufficient criterion is given for $B_4(\graf)$ to be formal as a space.
\end{remark}

\subsection{Small graphs are formal}

We begin with a classification of smooth small graphs---see Figure~\ref{figure: small graphs}.

\begin{lemma}\label{lem:small classification}
If $\graf$ is a small graph that is both connected and smooth, then $\graf$ is isomorphic to an isolated vertex, an interval $\intervalgraph$, a cycle $\cyclegraph{}$, a lollipop $\lollipopgraph{}$, a figure-eight $\figureeightgraph$, or a handcuff $\handcuffgraph$.
\end{lemma}
\begin{proof}
Any vertex of a small graph must be either isolated, $1$-valent, $2$-valent (with either two distinct edges or a self-loop), $3$-valent, with one self-loop, or $4$-valent, with $2$ self-loops.
A connected graph with a vertex which is isolated, $2$-valent with a self-loop, or $4$-valent with $2$ self-loops is necessarily an isolated vertex, the cycle graph $\cyclegraph{}$, or the figure-eight graph $\figureeightgraph$, respectively.
A graph containing a $2$-valent vertex with two distinct edges cannot be smooth.
Then in any remaining case, every vertex must be $1$-valent or $3$-valent with one self-loop.
A connected graph all of whose vertices are of these two kinds must have precisely two vertices; then there are three cases, namely the interval $\intervalgraph$, the lollipop graph $\lollipopgraph{}$, and the handcuff graph $\handcuffgraph$.
\end{proof}

This gives us most of what we need, since smoothings reflect formality.

\begin{lemma}\label{lem:smoothing formality}
Let $f:\graf\to \graf'$ be a smoothing.
If $\graf'$ is edge formal over $R$, then so is $\graf$.
\end{lemma}
\begin{proof}
The hypothesis guarantees the existence of the middle arrow in the diagram
\[S(\graf;R)\xrightarrow{S(f;R)}S(\graf';R)\dashrightarrow H_*(S(\graf';R))\xleftarrow{H_*(S(f;R))}H_*(S(\graf;R))\]
of quasi-isomorphisms of $R[E(\graf)]$-modules, where $R[E(\graf)]$ acts on the middle two entries by restriction.
\end{proof}

We can also eliminate self-loops by making choices.

\begin{lemma}
\label{lemma: reduce loops}
Let $\graf$ be a graph with a self-loop $e$ at the vertex $v$. Let $\graf_{\!-}$ be the graph obtained by replacing $e$ with an edge which is not a loop to a new vertex. Then the reduced \'{S}wi\k{a}tkowski complex $\fullyreduced{\graf;R}$ (unnaturally) decomposes as a direct sum of differential graded $R[E]$-modules as
\[
\fullyreduced{\graf_v;R}[1]\{1\}\oplus \fullyreduced{\graf_{\!-};R}\cong \fullyreduced{\graf;R}. 
\] In particular, $\graf$ is edge formal over $R$ if and only if both $\graf_v$ and $\graf_{\!-}$ are so.
\end{lemma}
\begin{proof}
Let $h_1$ and $h_2$ be the two half-edges of $e$. 
Write all half-edge generators at $v$ as differences $h_{1i}=h_1-h_i$
We obtain a direct sum decomposition of $\fullyreduced{\graf;R}$ by grouping generators (using this choice of presentation at $v$) according to whether or not they involve $h_{12}$. Since $h_{12}$ is closed, this decomposition respects the differential, and it respects the $R[E]$-action by inspection.

Now, there is a graph morphism $\iota$ from $\graf_{\!-}$ to $\graf$ taking the half-edge of $e$ to $h_1$, and we obtain a map $(\alpha,\beta)\mapsto h_{12}\alpha + \iota_*(\beta)$. It is an immediate verification that this map realizes the direct sum decomposition.
\end{proof}
Removing a vertex or turning a self-loop to an edge as in this lemma preserves smallness.

\begin{proof}[Proof of Theorem \ref{thm:formality}, ``if'' direction]
By Lemmas~\ref{lem:smoothing formality} and~\ref{lemma: reduce loops}, we may assume that $\graf$ is a smooth graph without self-loops. 
By Lemma~\ref{lem:small classification}, $\graf$ is a disjoint union of isolated vertices and intervals.
But then $\fullyreduced{\graf;R}$ has no differential, and there is nothing to prove.
\end{proof}

\subsection{Paradoxically decomposable cycles} 
We will derive the ``only if'' direction of Theorem~\ref{thm:formality} from the existence of certain cycles that we dub \emph{paradoxically decomposable}.

\begin{definition}
Let $A$ be a commutative ring, $I\subseteq A$ an ideal, $M$ a differential graded $A$-module, and $b\in M$ a cycle.
\begin{enumerate}
\item We say that $b$ is \emph{decomposable} with respect to $I$ if $b\in IM$. 
\item We say that $b$ is \emph{paradoxically decomposable} with respect to $I$ if 
\begin{enumerate}
\item $b$ is decomposable with respect to $I$ in $M$, and 
\item $[b]$ is \emph{not} decomposable with respect to $I$ in $H_*(M)$.
\end{enumerate}
\end{enumerate}
\end{definition}

\begin{lemma}\label{lemma: paradoxical}
Let $M$ be a differential graded $A$-module that is projective in each degree, $I\subseteq A$ an ideal, and $b\in M$ a cycle.
If $b$ is paradoxically decomposable with respect to $I$, then $M$ is not formal.
\end{lemma}
\begin{proof}
Assume that $M$ is formal.
Then, by our hypothesis on $M$ and Lemma~\ref{lemma:formality}, there is a map $f:M\to H_*(M)$ of $A$-modules inducing the identity on homology.
Since $b$ is decomposable with respect to $I$, we may write $b=\sum r_i b_i$ with $r_i\in I$, whence $f(b)=\sum r_i f(b_i)$; thus, $f(b)$ is decomposable with respect to $I$.
But $b$ is a cycle and $f$ induces the identity on homology, so $f(b)=[b]$, a contradiction.
\end{proof}

We briefly detour from our main discussion to present the following result, which is of independent interest, particularly from the point of view of \cite[Theorem 3.20]{AnDrummond-ColeKnudsen:SSGBG}.

\begin{proposition}\label{prop:tails formality}
Let $\graf$ be a graph with tails $e_1$, $e_2$, and $e_3$ in the same connected component, and let $R$ be $\mathbb{Z}$ or $\mathbb{Z}/2\mathbb{Z}$.
Then $S(\graf;R)$ is not formal over $R[e_1,e_2,e_3]$.
\end{proposition}
\begin{proof}
We will show that the $S(\graf;R)$ has a parodoxically decomposable cycle with respect to the ideal $I$ generated by $e_1$, $e_2$, and $e_3$. 

Because $e_1$, $e_2$, and $e_3$ lie in the same connected component, there is a subgraph $\graf_0\subseteq \graf$ containing these three edges and admitting a smoothing $\graf_0\to \stargraph{3}$. 
We will show that the corresponding star class
$\alpha\in H_*(B(\graf_0);R)\cong H_*(B(\stargraph{3});R)$ is not decomposable with respect to $I$, which will imply the claim, since a decomposable representative for this class is given by $\sum p_{ij}e_k$, where $p_{ij}$ is the unique element in $S(\graf_0;R)$ with $\partial p_{ij}=e_i-e_j$.

Recall the canonical homomorphism $\sigma: H_1(B_2(\graf);R)\to \mathbb{Z}/2\mathbb{Z}$ that records the permutation of the endpoints of a braid.
Since $\alpha$ is a star class, $\sigma(\alpha)=1$.
On the other hand, $\sigma(e\gamma)=0$ for any $e\in\{e_1, e_2, e_3\}$ and $\gamma\in H_1(B_1(\graf);R)$ because $e$ is a tail and thus is not part of the simple closed curve representing $\gamma$.
Then the assumption that $\alpha$ is decomposable with respect to $I$ implies the contradiction $\sigma(\alpha)=0$. 
\end{proof}
We expect this statement is true with arbitrary coefficients.

Next, we will use paradoxically decomposable cycles to establish non-formality in two basic cases, from which we will deduce the statement for a general large graph.
For the rest of this section, we shall only consider decomposability with respect to the ideal generated by $E$. 

We write $N$ for the number of essential vertices of $\graf$. As is well known, the homology of $B_k(\graf)$ vanishes above degree $N$; indeed, the \'{S}wi\k{a}tkowski complex itself vanishes in these degrees (see \cite[Theorem 3.3]{Ghrist:CSBGGR} for an earlier argument). Moreover,  this top homology is nonzero, since there can be no boundaries in degree $N$, and since the external product of standard star class representatives at the essential vertices of $\graf$ is a nonzero cycle in the \'{S}wi\k{a}tkowski complex. 

\begin{lemma}\label{lemma: trivalent is non-formal}
Let $\graf$ be a graph with no self-loops all of whose vertices are of valence $1$ or $3$.
If $\graf$ has a vertex of valence $3$, then $\graf$ is not edge formal.  
\end{lemma}
\begin{proof}
By~\cite[Proposition 5.25]{AnDrummond-ColeKnudsen:SSGBG}, $H_{N}(B_k(\graf);R)$ is $0$ for $k<2N$ and is one dimensional for $k=2N$, spanned by a $W$-torus, where $W$ is the set of trivalent vertices.
This class is not decomposable, since there is nothing in lower weight and degree $N$, but its standard representative is decomposable.
The claim now follows from Lemma \ref{lemma: paradoxical}.
\end{proof}

For the second result, we require the following preliminary notion:

\begin{definition}
We say that a vertex $v$ is {\em simple} if there are no self-loops at $v$ and no vertex $w$ with multiple edges between $v$ and $w$.
\end{definition}

\begin{lemma}\label{lemma: simple vertex is non-formal}
If $\graf$ has a simple essential vertex, then $\graf$ is not edge formal.
\end{lemma}
\begin{proof}
Let $v$ be a simple essential vertex of $\graf$, and assume without loss of generality that $\graf$ has no vertices of valence $2$.
Then $H_{N}(B(\graf);R)\neq0$, so $H_{N-1}(B(\graf_v);R)\neq0$ by Proposition \ref{proposition: les vertex reduced}.
Choose a nonzero class $\beta_v$ in the group $ H_{N-1}(B_k(\graf_v);R)$ with $k$ chosen minimally so that this group is nonzero, and write $\beta$ for the external product of $\beta_v$ and a star class at $v$.
Since $\beta$ is represented by a nonzero cycle $b$ in top degree, we conclude that $\beta\neq0$. There are no boundaries in top degree, so the cycle $b$ is necessarily unique, and it is decomposable, since the standard representative of the star class is so. Thus, it suffices to show that $\beta$ is not decomposable.

Applying Proposition \ref{proposition: les vertex reduced} at $v$, we obtain the following piece of the exact sequence
\[
0
\to H_{N}(B_{k+1}(\graf);R)\to
\bigoplus^{\ell-1} H_{N-1}(B_k(\graf_v;R))\xrightarrow{\delta} H_{N-1}(B_{k+1}(\graf_v;R))\] where $\ell\coloneqq{}d(v)$.
Since $\beta$ is a nonzero element of $H_{N}(B_{k+2}(\graf);R)$, the assumption of decomposability is the assumption that $\beta$ is a nontrivial $R[E]$-linear combination of elements from $H_{N}(B_{k+1}(\graf);R)$.
Therefore it will suffice to show that $H_{N}(B_{k+1}(\graf);R)$ vanishes, ie, that the connecting homomorphism $\delta$ is injective.
Since $\graf_v$ has $N-1$ essential vertices, this degree is the top homological degree for $B_k(\graf_v)$,
so homology and cycles coincide. It follows that the kernel of $\delta$ consists of tuples of cycles $(b_2,\ldots, b_\ell)$ in the weight $N$ summand of $S_{N-1}(\graf_v;R)$ such that 
\[
\sum_{j=2}^\ell b_j(e_j-e_1)=0
\]
where $e_j$ is the $j$th edge incident on $v$.
The $e_j$ are distinct by the simplicity of $v$.
We may assume without loss of generality that $b_2\neq0$.

Let $b_{j,1}$ be the chain obtained from $b_j$ by replacing $e_\ell$ with $e_1$.
Since the differential is $R[E]$-linear, each $b_{j,1}$ is again a cycle.
We claim further that $b_{j,1}$ is nonzero if $b_j$ was.
To see this, we note that the difference $b_j-b_{j,1}$ is divisible by $(e_\ell-e_1)$, so that, if $b_{j,1}=0$, we may conclude that $b_j$ is the product of $(e_\ell-e_1)$ and a \emph{cycle} of top degree in weight $k-1$, which is necessarily not a boundary.
Since $k$ was chosen to be minimal with respect to the existence of such a cycle, this is a contradiction.

Thus, the sum $\sum_{j=2}^{\ell-1} b_{j,1}(e_j-e_1)$ vanishes modulo $(e_\ell-e_1)$ and so vanishes, since none of the terms contain $e_\ell$ by construction.
Applying this procedure repeatedly, we obtain a nonzero cycle $b_{2,\ell-2}$ with $b_{2,\ell-2}(e_2-e_1)=0$, a contradiction.
 \end{proof}

\subsection{Large graphs are not formal}
In order to reduce the case of a general large graph to the cases already considered, we will make use of the following device:

\begin{definition}
Let $M_1$ be a differential graded $A_1$-module and $M_2$ be a differential graded $A_2$-module.
An {\em $(A_1,A_2)$-retraction} of $M_2$ onto $M_1$ consists of 
\begin{itemize}
\item a retraction of rings $A_1\xrightarrow{\iota} A_2\xrightarrow{\pi} A_1$, and
\item a retraction of $\mathbb{Z}$-modules $M_1\xrightarrow{i}M_2\xrightarrow{p} M_1$
\end{itemize}
 where $i$ is $A_1$-linear with respect to the $A_1$-module structure on $M_2$ induced by $\iota$ and $p$ is $A_2$-linear with respect to the $A_2$-module structure on $M_1$ induced by $\pi$. 
\end{definition}

The relation of this notion to questions of formality is the following:

\begin{lemma}\label{lemma: retracts}
Let $M_1$ and $M_2$ be differential graded $A_1$- and $A_2$-modules, respectively, and suppose that there exists an $(A_1,A_2)$-retraction of $M_2$ onto $M_1$.
If $M_2$ is projective in each degree and formal over $A_2$, then $M_1$ is formal over $A_1$.
\end{lemma}
\begin{proof}
Our hypotheses on $M_2$ and Lemma~\ref{lemma:formality} guarantee the existence of a map $M_2\to H_*(M_2)$ of differential graded $A_2$-modules inducing the identity on homology.
It follows that the map in homology induced by the composite \[M_1\to M_2\to H_*(M_2)\to H_*(M_1)\] is $H_*(p)\circ \mathrm{id}_{H_*(M_2)}\circ H_*(i)=\mathrm{id}_{H_*(M_1)}$, so it suffices to check that this map is $A_1$-linear, which follows from a diagram chase.
\end{proof}

\begin{definition} Let $\graf$ and $\graf'$ be graphs.
An \emph{algebraic retraction} (over $R$) from $\graf'$ to $\graf$ is a $(R[E(\graf)],R[E(\graf')])$-retraction of $\widetilde{S}(\graf';R)$ onto $\widetilde{S}(\graf;R)$.
In the presence of an algebraic retraction, we say that $\graf$ is an \emph{algebraic retract} of $\graf'$ (over $R$).
\end{definition}

The composition of algebraic retractions is an algebraic retraction.
An algebraic retraction need not arise from a topological retraction between configuration spaces. 

The ``only if'' direction of Theorem \ref{thm:formality} is an immediate consequence of the following result in combination with Lemmas \ref{lemma: trivalent is non-formal}, \ref{lemma: simple vertex is non-formal}, and \ref{lemma: retracts} (see Figure~\ref{figure: special large graphs} for an explanation of unfamiliar terminology).

\begin{lemma}\label{lemma: simple vertex or retract}
If $\graf$ is a smooth large graph, then $\graf$ has one of the following graphs as an algebraic retract: 
\begin{enumerate}
\item a graph with a simple essential vertex,
\item the theta graph $\thetagraph{3}$, or
\item the $A$ graph $\Agraph$.
\end{enumerate}
\end{lemma}

\begin{figure}[ht]
\begin{center}
\begin{tikzpicture}
\begin{scope}[xshift=3cm]
\fill[black] (-.25,0) circle (2.5pt);
\fill[black] (.25,0) circle (2.5pt);
\fill[black] (-.25,-.4) circle (2.5pt);
\fill[black] (.25,-.4) circle (2.5pt);
\draw(-.25,-.4) --(-.25,0) -- (.25,0) -- (.25,-.4);
\draw(-.25,0) .. controls (-.25,.6) and (.25,.6) .. (.25,0);
\draw(0,-.6) node[below]{$\Agraph$};
\end{scope}
\begin{scope}
\fill[black] (-.5,0) circle (2.5pt);
\fill[black] (.5,0) circle (2.5pt);
\draw(0,0) circle (.5cm);
\draw(.5,0)--(-.5,0);
\draw(0,-.6) node[below]{$\ \ \thetagraph{3}$};
\end{scope}
\end{tikzpicture}
\end{center}
\caption{The theta graph and the $A$ graph}\label{figure: special large graphs}
\end{figure}
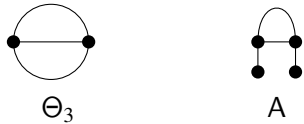

The proof of this result relies on the fact that algebraic retractions arise geometrically as the result of \emph{surgeries} on graphs (see Figure \ref{figure: surgeries}).

\begin{lemma}\label{lemma: surgery}
Let $\graf$ and $\graf'$ be graphs, and suppose that $\graf$ is obtained from $\graf'$ by replacing a connected subgraph $\graphfont{\Delta}$, attached to the rest of $\graf$ only at the vertices $v_1\neq v_2$, with a single edge $e_0$ between $v_1$ and $v_2$, ie, 
\[
\graf' = (\graf\setminus e_0) \coprod_{v_1\cup v_2} \graphfont{\Delta}
\]
Then $\graf$ is an algebraic retract of $\graf'$ over any $R$.
\end{lemma}
We call this operation a \emph{surgery}.
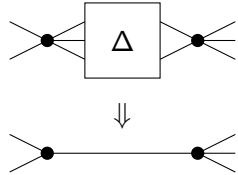
\begin{figure}[ht]
\begin{center}
\begin{tikzpicture}
\begin{scope}[xshift=6.5cm]
\begin{scope}[xshift=0cm]
\fill[black] (0,.5) circle (2.5pt);
\draw(0,.5) -- (-.5,.25);
\draw(0,.5) -- (-.5,.75);
\draw(0,.5) -- (.5,.25);
\draw(0,.5) -- (.5,.5);
\draw(0,.5) -- (.5,.75);
\draw (.5,0) rectangle (1.5,1);
\fill[black] (2,.5) circle (2.5pt);
\draw(2,.5) -- (1.5,.25);
\draw(2,.5) -- (1.5,.75);
\draw(2,.5) -- (2.5,.25);
\draw(2,.5) -- (2.5,.5);
\draw(2,.5) -- (2.5,.75);
\draw(1,.5)node{$\graphfont{\Delta}$};
\end{scope}
\draw(1,-.5)node{$\Downarrow$};
\begin{scope}[yshift=-1.5cm]
\fill[black] (0,.5) circle (2.5pt);
\draw(0,.5) -- (-.5,.25);
\draw(0,.5) -- (2,.5);
\draw(0,.5) -- (-.5,.75);
\fill[black] (2,.5) circle (2.5pt);
\draw(2,.5) -- (2.5,.25);
\draw(2,.5) -- (2.5,.75);
\draw(2,.5) -- (2.5,.5);
\end{scope}
\end{scope}
\end{tikzpicture}
\end{center}
\caption{Depiction of a surgery}\label{figure: surgeries}
\end{figure}
Assuming this result momentarily, we complete the proof of Theorem \ref{thm:formality}.

\begin{proof}[Proof of Lemma~\ref{lemma: simple vertex or retract}]
A connected component is an algebraic retract and largeness is a local property so it suffices to consider the connected case.
Suppose $\graf$ has a vertex of valence strictly greater than three with a self-loop.
The direct sum decomposition of Lemma~\ref{lemma: reduce loops} gives an algebraic retraction from $\graf$ to the graph $\graf_{\!-}$ obtained by replacing a self-loop with a tail. 
The graph $\graf_{\!-}$ is smooth, connected, and large if $\graf$ was.
Suppose $\graf$ has a vertex of valence three with a self-loop. 
There is a surgery replacing the closed star of the vertex with a tail which preserves smoothness and largeness.
Then iteratively we may assume $\graf$ contains no self-loops.

Suppose that $\graf$ has a pair of vertices sharing three edges. 
We obtain an algebraic retraction onto $\thetagraph{3}$ using Lemma \ref{lemma: surgery} with $\graphfont{\Delta}$ the complement in $\graf$ of two of these (open) edges. 
See Figure~\ref{figure: surgeries lemma}(a).
Suppose that $\graf$ has a pair of vertices $v$ and $w$ sharing two edges $e_1$ and $e_2$. 
Since $\graf$ is smooth, each of $v$ and $w$ has a third edge.
If there is a path of edges disjoint from $e_1$ and $e_2$ connecting $v$ and $w$, then surgery produces an algebraic retraction onto $\thetagraph{3}$. 
See Figure~\ref{figure: surgeries lemma}(b).
If not, two surgeries produce an algebraic retraction onto $\Agraph$.
See Figure~\ref{figure: surgeries lemma}(c).

Then if there is no algebraic retraction from $\graf$ onto $\thetagraph{3}$ or $\Agraph$, the removal of self-loops constitutes an algebraic retraction onto a connected smooth large graph with a simple essential vertex.
\end{proof}

\begin{figure}[ht]
\begin{center}
\begin{tikzpicture}
\begin{scope}[xshift=0cm]
\begin{scope}[xshift=0cm]
\fill[black] (0,.5) circle (2.5pt);
\fill[black] (2,.5) circle (2.5pt);
\draw(0,.5) -- (2,.5);
\draw(0,.5) to [bend right=50] (2,.5);
\draw(0,.5) to [bend right=-50] (2,.5);
\draw(0,.5) to [bend right=50] (.5,-.7);
\draw(0,.5) to [bend right=30] (.5,-.3);
\draw(2,.5) to [bend right=-50] (1.5,-.7);
\draw[fill=white] (.5,.25) rectangle (1.5,-.75);
\draw[dashed](0,.5) to [bend right=50] (2,.5);
\draw(1,-.25)node{$\graphfont{\Delta}$};
\end{scope}
\draw(1,-1.15)node{$\Downarrow$};
\begin{scope}[yshift=-2.5cm]
\fill[black] (0,.5) circle (2.5pt);
\fill[black] (2,.5) circle (2.5pt);
\draw(0,.5) -- (2,.5);
\draw(0,.5) to [bend right=50] (2,.5);
\draw(0,.5) to [bend right=-50] (2,.5);
\draw(1,-.6) node{(a)};
\end{scope}
\end{scope}
\begin{scope}[xshift=3.5cm]
\begin{scope}[xshift=0cm]
\fill[black] (0,.5) circle (2.5pt);
\draw(0,.5) .. controls (.4,1.5) and (1.6,1.5) .. (2,.5);
\draw(0,.5) .. controls (.4,-.5) and (1.6,-.5) .. (2,.5);
\draw(0,.5) -- (.5,.25);
\draw(0,.5) -- (.5,.5);
\draw(0,.5) -- (.5,.75);
\draw (.5,0) rectangle (1.5,1);
\fill[black] (2,.5) circle (2.5pt);
\draw(2,.5) -- (1.5,.25);
\draw(2,.5) -- (1.5,.75);
\draw(1,.5)node{$\graphfont{\Delta}$};
\end{scope}
\draw(1,-.75)node{$\Downarrow$};
\begin{scope}[yshift=-2.5cm]
\fill[black] (0,.5) circle (2.5pt);
\draw(0,.5) .. controls (.4,1.5) and (1.6,1.5) .. (2,.5);
\draw(0,.5) .. controls (.4,-.5) and (1.6,-.5) .. (2,.5);
\draw(0,.5) -- (2,.5);
\fill[black] (2,.5) circle (2.5pt);
\draw(1,-.6) node{(b)};
\end{scope}
\end{scope}
\begin{scope}[xshift=7cm]
\begin{scope}[xshift=0cm]
\fill[black] (.5,.5) circle (2.5pt);
\fill[black] (1.5,.5) circle (2.5pt);
\fill[black] (-.5,.5) circle (2.5pt);
\fill[black] (2.5,.5) circle (2.5pt);
\draw(.5,.5) to[bend right=50] (1.5,.5);
\draw(.5,.5) to[bend right=-50] (1.5,.5);
\draw(-.5,.5) -- (.5,.5);
\draw(2.5,.5) to[bend right=-50] (1.5,.5);
\draw(2.5,.5) to[bend right=50] (1.5,.5);
\draw(1.5,.5) -- (1.7,.5);
\draw[fill=white] (.3,.2) rectangle (-.3,.8);
\draw[fill=white] (2.3,.2) rectangle (1.7,.8);
\draw(0,.5)node{$\graphfont{\Delta}_1$};
\draw(2,.5)node{$\graphfont{\Delta}_2$};
\end{scope}
\draw(1,-.125)node{$\Downarrow$};
\begin{scope}[yshift=-1.25cm]
\fill[black] (.5,.5) circle (2.5pt);
\fill[black] (1.5,.5) circle (2.5pt);
\fill[black] (-.5,.5) circle (2.5pt);
\fill[black] (2.5,.5) circle (2.5pt);
\draw(.5,.5) to[bend right=50] (1.5,.5);
\draw(.5,.5) to[bend right=-50] (1.5,.5);
\draw(-.5,.5) -- (.5,.5);
\draw(2.5,.5) to[bend right=-50] (1.5,.5);
\draw(2.5,.5) to[bend right=50] (1.5,.5);
\draw(1.5,.5) -- (1.7,.5);
\draw[fill=white] (2.3,.2) rectangle (1.7,.8);
\draw(2,.5)node{$\graphfont{\Delta}_2$};
\draw(1,-.125)node{$\Downarrow$};
\end{scope}
\begin{scope}[yshift=-2.5cm]
\fill[black] (.5,.5) circle (2.5pt);
\fill[black] (1.5,.5) circle (2.5pt);
\fill[black] (-.5,.5) circle (2.5pt);
\fill[black] (2.5,.5) circle (2.5pt);
\draw(.5,.5) to[bend right=50] (1.5,.5);
\draw(.5,.5) to[bend right=-50] (1.5,.5);
\draw(-.5,.5) -- (.5,.5);
\draw(2.5,.5) -- (1.5,.5);
\draw(1,-.6) node{(c)};
\end{scope}
\end{scope}
\end{tikzpicture}
\end{center}
\caption{Surgeries for the proof of Lemma~\ref{lemma: simple vertex or retract}}\label{figure: surgeries lemma}
\end{figure}
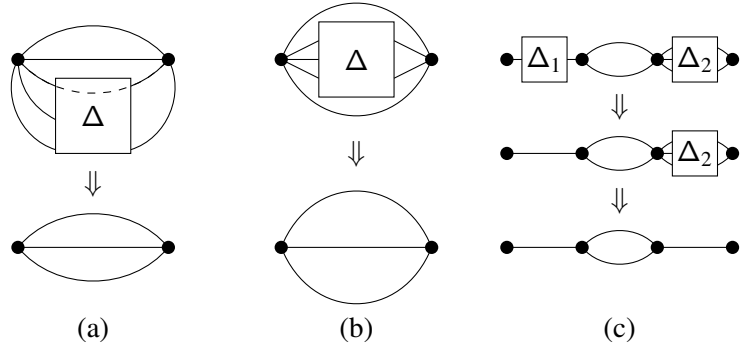

To prepare for the proof of the surgery lemma, we observe that the \'{S}wi\k{a}tkowski complex
enjoys a certain non-canonical functoriality for subdivisions.
More precisely, the map induced by a smoothing $f:\graf_1\to \graf_2$ admits a non-canonical left inverse, which we regard as corresponding to the subdivision $f^{-1}$.
It suffices to construct this morphism in the case of a subdivision of a single edge $e$ with half-edges $h_1$ and $h_2$ into two edges $e_1$ and $e_2$ meeting at the common vertex $v$ with corresponding half-edges $h_{v1}$ and $h_{v2}$.
We define the map in this case by the assignments $e\mapsto e_1$, $h_1\mapsto h_1$, and $h_2\mapsto h_{v1}-h_{v2}+h_2$. 
See Figure~\ref{fig: noncanonical left inverse}.
This map is a map of differential bigraded $R[E_2]$-modules, where the action of $R[E_2]$ on the target is via the map $R[E_2]\to R[E_1]$ just constructed and hence also non-canonical.
The same remarks hold for reduced complexes.
\begin{figure}[ht]
\begin{center}
\begin{tikzpicture}
\begin{scope}[xshift=0cm,yshift=-1.5cm]
\fill[black] (0,0) circle (2.5pt);
\fill[black] (2,0) circle (2.5pt);
\fill[black] (-2,0) circle (2.5pt);
\draw (-2,0)--(2,0);
\draw[line width=2pt] (.8,0)--(-.8,0);
\draw[line width=2pt] (1.2,0)--(2,0);
\draw[line width=2pt] (-1.2,0)--(-2,0);
\draw(-2,0)node[left]{$v_1$};
\draw(2,0)node[right]{$v_2$};
\draw(.4,0)node[below]{$h_{v2}$};
\draw(-.4,0)node[below]{$h_{v1}$};
\draw(1.6,0)node[below]{$h_{2}$};
\draw(-1.6,0)node[below]{$h_{1}$};
\end{scope}
\begin{scope}[yshift=0]
\draw[dashed,->,shorten >=5pt,shorten <=5pt] (2,0)--(2,-1.5);
\draw[dashed,->,shorten >=5pt,shorten <=5pt] (-2,0)--(-2,-1.5);
\draw[dashed,->,shorten >=5pt,shorten <=5pt] (.4,0)--(-.8,-1.5);
\draw[dashed,->,shorten >=5pt,shorten <=5pt] (-.4,0)--(-1.2,-1.5);
\fill[black] (2,0) circle (2.5pt);
\fill[black] (-2,0) circle (2.5pt);
\draw (-2,0)--(2,0);
\draw[line width=2pt](.4,0)--(2,0);
\draw[line width=2pt](-.4,0)--(-2,0);
\draw(-2,0)node[left]{$v_1$};
\draw(1.2,0)node[above]{$h_2$};
\draw(-1.2,0)node[above]{$h_1$};
\draw(2,0)node[right]{$v_2$};
\end{scope}
\end{tikzpicture}
\end{center}
\caption{A choice of left inverse for the map induced by a smoothing}
\label{fig: noncanonical left inverse}
\end{figure}
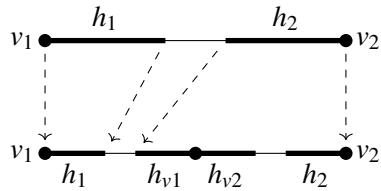

\begin{proof}[Proof of Lemma \ref{lemma: surgery}]
There is a (non-canonical) graph embedding of a subdivision of $\graf$ into $\graf'$, where only the edge $e_0$ is subdivided.
Thus, for any surgery, there is a non-canonical map $\widetilde{S}(\graf;R)\to \widetilde{S}(\graf';R)$.
This non-canonical map admits a left inverse $r$, defined as follows. Write $\graf'$ as $(\graf\setminus e_0) \coprod_{v_1\cup v_2}\graphfont{\Delta}$. Then:

\begin{align*}
r(h)&=\begin{cases}
h_i & h\in H(\graphfont{\Delta}),\, v(h)=v_i\\
0 & h\in H(\graphfont{\Delta}),\, v(h)\neq v_i\\
h & h\not\in H(\graphfont{\Delta})
\end{cases}&r(e)&=\begin{cases}
e_0&e\in E(\graphfont{\Delta})\\
e&e\not\in E(\graphfont{\Delta}).
\end{cases}
\end{align*}

The map on edges endows $\widetilde{S}(\graf;R)$ with a $R[E(\graf')]$-module structure, and linearity is a direct verification.
By inspection, these maps constitute a retraction both at the level of modules and at the level of rings.
\end{proof}

\section{Two isomorphisms}\label{section:two isomorphisms}

We return to the proof of Theorem \ref{thm:module iso}.
Along the way, we will encounter two very different routes to the \'{S}wi\k{a}tkowski complex.
In Section \ref{section:functorial model}, we review the method of \cite{AnDrummond-ColeKnudsen:SSGBG}, which has the advantage that it naturally outputs a functor on the category $\Gph$.
On the other hand, there is the cubical deformation retract introduced by \'{S}wi\k{a}tkowski \cite{Swiatkowski:EHDCSG}, which offers a more direct and geometric comparison to $B(\graf)$ at the cost of non-functorial choices.
Our argument will combine the advantages of these two approaches.

In this section and the next, we write $C^\sing(X)$ for the singular chain complex of the topological space $X$. 
If $X$ is a CW complex, we denote the cellular chain complex of $X$ by $C(X)$. 
We will make use of the existence of a zig-zag 
\[
C^\sing\xleftarrow{\sim}\bullet\xrightarrow{\sim} C
\] 
of quasi-isomorphisms connecting these complexes, which is natural for cellular maps.
The specifics of the intermediate object will play no role here, but see \cite[Construction 2.22]{AnDrummond-ColeKnudsen:SSGBG} for one option.

\subsection{Functorial model}\label{section:functorial model} We now recall some of the work of \cite{AnDrummond-ColeKnudsen:SSGBG}.
The starting point is a cell complex introduced by Abrams \cite[Definition 2.1]{Abrams:CSBGG}.

\begin{definition}
Let $X$ be a cell complex.
The $k$th unordered \emph{configuration complex} of $X$ is the subspace 
\[
B_k^\Box(X)=\left(\bigcup_{\overline{c}_i\cap \overline{c}_j=\varnothing} c_1\times\cdots \times c_k\right)_{\!\!\!\mbox{$/\Sigma_k$}}\subseteq B_k(X)
\] where the union is taken over the set of $k$-tuples of disjoint open cells of $X$.
\end{definition}

The configuration complexes of $X$ are approximations to the configuration spaces of $X$ by cell complexes, and the accuracy of this approximation improves with subdivision.
Following \cite[\textsection 3]{FarleySabalka:DMTGBG}, we say that a graph $\graf$ is \emph{sufficiently subdivided for $k$} if, first, every path in $\graf$ between essential vertices passes through at least $k-1$ edges, and, second, every loop passes through at least $k+1$ edges.

\begin{theorem}[{\cite[Theorem 2.1]{Abrams:CSBGG}}]\label{thm:abrams}
Let $\graf\to \graf'$ be a subdivision with $\graf'$ sufficiently subdivided for $k$.
The inclusion $B_k^\Box(\graf')\to B_k(\graf)$ is a deformation retract.
\end{theorem}

In order to work functorially and with all $k$ at once, it is more convenient to consider all subdivisions simultaneously.
Denoting by $\subd$ the set of subdivisions of $\graf$, viewed as a category under refinement, Theorem \ref{thm:abrams} implies that the natural map \[\colim_{\graf'\in \subd}B_k^\Box(\graf')\to B_k(\graf)\] is a weak homotopy equivalence \cite[Theorem 2.8]{AnDrummond-ColeKnudsen:SSGBG}.

Including arbitrary subdivisions has the further benefit of allowing one to work locally on $\graf$.
In order to leverage this flexibility, we return to the heuristic interpretation of Remark \ref{remark:heuristic}, in which generators of $\intrinsic{\graf}$ represent ``states'' in $B(\graf)$ obtained by prescribing ``local states'' near vertices and along edges, which are compatible on half-edges.

More precisely, define $\starreduced{\stargraph{n}}$ to be the subcomplex of $S(\stargraph{n})$ spanned over $\mathbb{Z}[E]$ by the symbols $\varnothing$, the star vertex and half-edges at the star vertex.
For $e$ an edge, viewed as a $1$-cell, define $\starreduced{e}$ as the subcomplex of $S(e)$ spanned by basis elements with no vertices or half-edges (then $\starreduced{e}$ is canonically isomorphic to $\mathbb{Z}[e]$). For these atomic graphs, one can write down a map $C(B^\Box(\graf'))\to \starreduced{\graf}$ and check by hand that it is a quasi-isomorphism for sufficiently fine subdivisions---see \cite[\textsection 4.2--4.3]{AnDrummond-ColeKnudsen:SSGBG}.

Now, for $\Xigraph$ a disjoint union of star graphs and edges, define $\starreduced{\Xigraph}$ to be the tensor product over the connected components of $\Xigraph$ of the corresponding subcomplexes defined above. 
Then there is an isomorphism 
\[
S(\graf)\cong  \starreduced{\coprod_{v\in V}\stargraph{d(v)}}\bigotimes_{\starreduced{\left(\coprod_{e\in E} \intervalgraph\right)\times_E H}}\starreduced{\coprod_{e\in E} \intervalgraph}.
\]

We now introduce a device that aids in piecing together these local identifications.
Recall that each edge of $\graf$ is identified with $(0,5)$ via its parametrization. We define a map $\pi:\graf\to [0,1]$ by setting \[\pi(t)=\begin{cases}
t-1&\quad t\in [1,2]\subseteq e\\
1&\quad t\in [2,3]\subseteq e\\
4-t&\quad t\in [3,4]\subseteq e\\
0&\quad \text{otherwise},\end{cases}\] and by sending every vertex of $\graf$ to $0$.

\begin{definition}
\label{definition: gap}
A \emph{gap} in $\graf$ is a subspace of the form $A=\pi^{-1}(A_0)$, where $A_0\subseteq [0,1]$ is a nonempty open subset such that \begin{enumerate}
\item the complement $[0,1]\setminus A_0$ is a (possibly empty) finite union of closed intervals of positive length, and
\item if $i\in \{0, 1\}$ lies in the closure of $A_0$, then $i\in A_0$.
\end{enumerate}
See Figure~\ref{figure: gap}.
We write $\gaps$ for the poset of gaps. 
\end{definition}

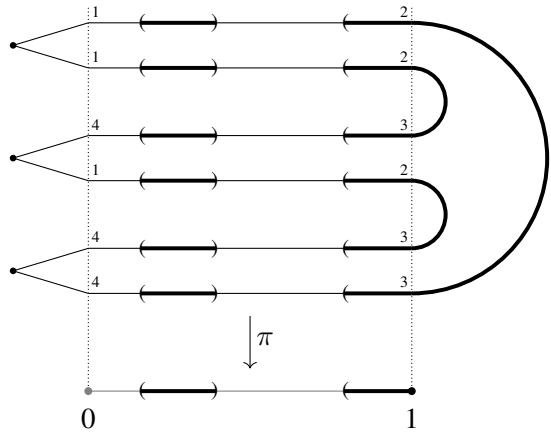
\begin{figure}[ht]
\begin{center}
\begin{tikzpicture}
\begin{scope}[xshift=.3cm]
\draw[line width=1.5] (5,-1.2) arc [start angle=-90, end angle=90, radius=.45];
\draw[line width=1.5] (5,-2.7) arc [start angle=-90, end angle=90, radius=.45];
\draw[line width=1.5] (5,-3.3) arc [start angle=-90, end angle=90, radius=1.8];
\draw[densely dotted] (.7,.5)--(.7,-4.6);
\draw[densely dotted] (5,.5)--(5,-4.6);
\end{scope}
\fill[black] (0,0) circle (1.25pt);
\draw(0,0) -- (1,.3);
\draw(0,0) -- (1,-.3);
\draw(1,.3)--(5.3,.3);
\draw(1,-.3)--(5.3,-.3);
\draw[line width=1.5](1.7,-.3)--(2.7,-.3);
\draw[line width=1.5](4.4,-.3)--(5.3,-.3);
\draw[line width=1.5](1.7,.3)--(2.7,.3);
\draw[line width=1.5](4.4,.3)--(5.3,.3);
\draw(1.1,.45) node{$\scriptscriptstyle 1$};
\draw(1.1,-.15) node{$\scriptscriptstyle 1$};
\draw(5.2,.45) node{$\scriptscriptstyle 2$};
\draw(5.2,-.15) node{$\scriptscriptstyle 2$};
\draw(1.72,-.3)node{$\scriptstyle($};
\draw(2.68,-.3)node{$\scriptstyle)$};
\draw(4.42,-.3)node{$\scriptstyle($};
\draw(1.72,.3)node{$\scriptstyle($};
\draw(2.68,.3)node{$\scriptstyle)$};
\draw(4.42,.3)node{$\scriptstyle($};
\begin{scope}[yshift=-1.5cm]
\fill[black] (0,0) circle (1.25pt);
\draw(0,0) -- (1,.3);
\draw(0,0) -- (1,-.3);
\draw(1,.3)--(5.3,.3);
\draw(1,-.3)--(5.3,-.3);
\draw[line width=1.5](1.7,-.3)--(2.7,-.3);
\draw[line width=1.5](4.4,-.3)--(5.3,-.3);
\draw[line width=1.5](1.7,.3)--(2.7,.3);
\draw[line width=1.5](4.4,.3)--(5.3,.3);
\draw(1.1,.45) node{$\scriptscriptstyle 4$};
\draw(1.1,-.15) node{$\scriptscriptstyle 1$};
\draw(5.2,.45) node{$\scriptscriptstyle 3$};
\draw(5.2,-.15) node{$\scriptscriptstyle 2$};
\draw(1.72,-.3)node{$\scriptstyle($};
\draw(2.68,-.3)node{$\scriptstyle)$};
\draw(4.42,-.3)node{$\scriptstyle($};
\draw(1.72,.3)node{$\scriptstyle($};
\draw(2.68,.3)node{$\scriptstyle)$};
\draw(4.42,.3)node{$\scriptstyle($};
\end{scope}
\begin{scope}[yshift=-3cm]
\fill[black] (0,0) circle (1.25pt);
\draw(0,0) -- (1,.3);
\draw(0,0) -- (1,-.3);
\draw(1,.3)--(5.3,.3);
\draw(1,-.3)--(5.3,-.3);
\draw[line width=1.5](1.7,-.3)--(2.7,-.3);
\draw[line width=1.5](4.4,-.3)--(5.3,-.3);
\draw[line width=1.5](1.7,.3)--(2.7,.3);
\draw[line width=1.5](4.4,.3)--(5.3,.3);
\draw(1.1,.45) node{$\scriptscriptstyle 4$};
\draw(1.1,-.15) node{$\scriptscriptstyle 4$};
\draw(5.2,.45) node{$\scriptscriptstyle 3$};
\draw(5.2,-.15) node{$\scriptscriptstyle 3$};
\draw(1.72,-.3)node{$\scriptstyle($};
\draw(2.68,-.3)node{$\scriptstyle)$};
\draw(4.42,-.3)node{$\scriptstyle($};
\draw(1.72,.3)node{$\scriptstyle($};
\draw(2.68,.3)node{$\scriptstyle)$};
\draw(4.42,.3)node{$\scriptstyle($};
\end{scope}
\begin{scope}[yshift=-1.8cm]
\draw [->](3.15,-1.8)--(3.15,-2.5);
\draw(3.1,-2.1)node[right]{$\pi$};
\draw[gray](1,-2.8)--(5.3,-2.8);
\draw(1.72,-2.8)node{$\scriptstyle($};
\draw(2.68,-2.8)node{$\scriptstyle)$};
\draw(4.42,-2.8)node{$\scriptstyle($};
\draw[line width=1.5]
(1.7,-2.8)--(2.7,-2.8);
\draw[line width=1.5](4.4,-2.8)--(5.3,-2.8);
\draw(1,-2.9) node[below]{$0$};
\draw(5.3,-2.9) node[below]{$1$};
\fill[gray] (1,-2.8) circle (1.5pt);
\fill (5.3,-2.8) circle (1.5pt);
\end{scope}
\end{tikzpicture}
\end{center}
\caption{A gap in the complete graph $\completegraph{3}$}
\label{figure: gap}
\end{figure}
Since the complement of a gap is a disjoint union of stars and intervals, we have a map $C(B^\Box(\graf'\setminus A))\to \starreduced{\graf\setminus A}$ whenever $A$ is a gap and $\graf'$ is a subdivision in which $A$ is a union of cells.
In this way, we obtain the following zig-zag of quasi-isomorphisms:
\[
\begin{tikzcd}[row sep=1.2em]
\displaystyle\colim_{\graf'\in\subd} C^\sing(B^\Box(\graf'))\ar[r,"\sim"]&C^\sing(B(\graf))\\
\bullet\ar[u,"\wr"']\ar[d,"\wr"]\\
\displaystyle\colim_{\graf'\in\subd} C(B^\Box(\graf'))\\
\displaystyle\hocolim_{A\in \gaps^{op}}\colim_{\graf'\in\subd_A} C(B^\Box(\graf'\setminus A))\ar[d,"\wr"]\ar[u,"\wr"']\\
\displaystyle\hocolim_{A\in \gaps^{op}}\starreduced{\graf\setminus A}\ar[r,"\sim"]&S(\graf).
\end{tikzcd}
\] 
Here, $\subd_A\subseteq \subd$ contains only those subdivisions of $\graf$ in which $A$ is a union of cells.
For details on why these maps are quasi-isomorphisms and why the resulting isomorphism on homology is functorial, see \cite[\textsection 4]{AnDrummond-ColeKnudsen:SSGBG}.

\begin{theorem}[{\cite[Theorem~4.5]{AnDrummond-ColeKnudsen:SSGBG}}]\label{thm:ADCK}
There is a natural isomorphism \[H_*(B(\graf))\cong H_*(\intrinsic{\graf})\] of functors from $\Gph$ to bigraded Abelian groups.
\end{theorem}

\subsection{Cubical model} 
We now recall the cubical model of $B(\graf)$ introduced in \cite{Swiatkowski:EHDCSG} and corrected in~\cite[\textsection 2.1]{ChettihLuetgehetmann:HCSTL}. Given a graph $\graf$, we write $A(\graf)$ for the set of labelings \[\lambda: E\sqcup V\to \mathbb{Z}_{\geq0}\sqcup V\sqcup H\sqcup \{\varnothing\}\] such that $\lambda(e)\in \mathbb{Z}_{\geq0}$ and $\lambda(v)\in \{\varnothing, v\}\sqcup H(v)$.

\begin{construction}[\'{S}wi\k{a}tkowski]\label{construction:cubical model}
Define a space $UK(\graf)$ as the quotient \[UK(\graf)=\faktor{\coprod_{\lambda\in  A(\graf)}\{\lambda\}\times [0,1]^{\lambda^{-1}(H)}}{\sim},\] 
where the equivalence relation is determined as follows. 
Fix a half-edge $h$ with edge $e$ and vertex $v$. Suppose $\lambda$, $\lambda_{0}$, and $\lambda_{1}$ are labellings which agree except on $v$ and $e$, where 
\begin{align*}
\lambda(v)&=h 
&
\lambda_{0}(v)&=v 
&\lambda_{1}(v)&=\varnothing 
\\
\lambda(e)&=n
& \lambda_{0}(e)&=n
&
\lambda_{1}(e)&=n+1.
\end{align*} 
Then we glue according to the following identifications for $\epsilon\in \{0,1\}$:
\begin{align*}
\{\lambda_{\epsilon}\}\times[0,1]^{\lambda_{\epsilon}^{-1}(H)}
&\cong 
\{\lambda\}\times [0,1]^{\lambda^{-1}(H\setminus \{h\})}\times\{\epsilon\}
\end{align*} 
as a subset of $\{\lambda\}\times [0,1]^{\lambda^{-1}(H)}$.
\end{construction}

To specify a point in $UK(\graf)$, it is enough to give a labeling $\lambda\in A(\graf)$ together with numbers $t(h)\in[0,1]$ for each $h\in \lambda^{-1}(H)$.

We endow the set $A(\graf)$ with a bigrading by declaring that 
\[|\lambda|=\left(|\lambda^{-1}(H)|,\, \sum_{E}\lambda(e)+|\lambda^{-1}(V)|+|\lambda^{-1}(H)|
\right)\] 
and note that $UK(\graf)$ splits as a disjoint union of cell complexes $UK_k(\graf)$ whose $i$-cells are those in bigrading $(i,k)$.

\begin{observation}\label{obs:swiatkowski chains}
By inspection, the map \begin{align*}
C(UK(\graf))&\to S(\graf)&
\lambda&\mapsto \left(\prod_{E}e^{\lambda(e)}\right)\otimes\bigotimes_V\lambda(v)
\end{align*} is a bigraded chain isomorphism. 
\end{observation}

\begin{remark}
To make Observation~\ref{obs:swiatkowski chains} precise, we should take care with orientations. 
One way to specify an orientation on a cube which is a product of intervals indexed by a set $J$ is to give an order on $J$ up to even permutation. 
One way to specify the correct sign on a tensor product of odd degree vector spaces indexed by a set $J$ is to give an order on $J$ up to even permutation.
Then, with appropriate conventions as to which is the ``positive'' and which is the ``negative'' end of an interval, the map of the observation intertwines these conventions.
\end{remark}
\begin{remark}
In previous work along these lines~\cite{ChettihLuetgehetmann:HCSTL,Luetgehetmann:CSG,Swiatkowski:EHDCSG}, neither $2$-valent nor $1$-valent vertices were considered, but there is no obstruction to this mild generalization.
\end{remark}

We now show that the complex $UK(\graf)$ is homotopy equivalent to the configuration space of interest.
Apart from slight modifications for simplicity and compatibility with our setup, this argument is essentially that of~\cite[\textsection 2.1]{ChettihLuetgehetmann:HCSTL}.%
\footnote{%
As observed by L\"utgehetmann, the map defined by \'Swi\k{a}tkowski is not quite a retraction \cite[p. 24]{Luetgehetmann:CSG}.
Unfortunately, the replacement retraction constructed by L\"utgehetmann is not continuous, as one can see by comparing the formulas for $t_x(s)$ appearing on p.~22 for $k=2$ and $k=3$ in a situation in which the first of three particles in an edge approaches the initial endpoint of that edge.
Chettih--L\"utgehetmann give a continuous retraction.
}
In particular, our homotopy equivalence will be homotopic to the one considered in that work.

In order to compare $B(\graf)$ and $UK(\graf)$, we first deform $B(\graf)$ onto the subspace $\widetilde B(\graf)$ of configurations $x$ with the property that, for each vertex $v$, at most one coordinate lies in the open star $\st{1}(v)$ of radius 1 at $v$.
This deformation is achieved by radial expansion in each star simultaneously.
Having taken this intermediate step, we define our comparison map $\rho$ as the composite \[\rho:B(\graf)\to \widetilde B(\graf)\to UK(\graf),\] where the second map records the presence or absence of particles at vertices, the coordinate of any particle in $\st{1}(v)\setminus \{v\}\cong H(v)\times(0,1)$, and the number of particles lying in the subinterval $[1,4]$ of each edge.
The details of the deformation are given below in Construction \ref{construction:deformation}; for now, we state the following result concerning $\rho$.

\begin{proposition}\label{prop:rho homotopy equivalence}
The map $\rho$ is a homotopy equivalence.
\end{proposition}

\begin{construction}\label{construction:deformation}
We define the deformation of $B(\graf)$ onto $\widetilde B(\graf)$ as follows.
For each vertex $v$, we have $\overline{\st{2}(v)}\setminus\{v\}\cong H(v)\times(0,2]$.
On each of these intervals, for a fixed configuration $x$, we use the homotopy \[(s,t)\mapsto 2\left(\frac{1}{2}s \right)^{e^{-t}}
\qquad\qquad\vcenter{\hbox{\def\svgscale{0.4}\input{deformation_input.tex}}}
\] on all points of $x\cap( \{h\}\times(0,2])$ simultaneously.
Here, $s\in(0,2]$ and $t\in[0,t(x, v)]$, where $t(x,v)$ is the least $t$ such that the resulting configuration has at most one particle in $\st{1}(v)$.%
\footnote{Explicitly, $t(x,v)$ is $0$ if there are fewer than $2$ points in $\st{2}(v)$ and otherwise $\log(1-\log_2(s_v))$ where $s_v$ is the distance from $v$ to the second closest point in $\st{2}(v)$ in $x$.}
Since $t(x,v)$ is continuous in $x$, and since $t(x,v)=0$ for $x\in \widetilde B(\graf)$, this prescription defines a deformation retraction.
\end{construction}

\begin{proof}[Proof of Proposition \ref{prop:rho homotopy equivalence}]
Chettih--L\"utgehetmann~\cite[\textsection 2.1]{ChettihLuetgehetmann:HCSTL} define a cube complex $\widetilde{UK}_k(\graf)$ which is a deformation retract of the \emph{ordered} configuration space of $k$ points in $\graf$. 
Their cube complex is a $\Sigma_k$-cover of the weight $k$ subcomplex of $UK(\graf)$.
Their deformation retraction is equivariant and so passes to a deformation retraction of $B(\graf)$ onto $UK(\graf)$.
The composite $B(\graf)\xrightarrow{\rho} UK(\graf)\to B(\graf)$ with the inclusion of this quotient deformation retract differs from the identity only in the positions of particles in individual open edges.
Therefore, the two maps are homotopic by edgewise straight line homotopies. 
It follows that $\rho$ is a one-sided homotopy inverse to a homotopy equivalence and hence itself a homotopy equivalence.
\end{proof}

\subsection{Comparison of models} 
We have two isomorphisms $H_*(B(\graf))\cong H_*(S(\graf))$.
The isomorphism of Theorem~\ref{thm:ADCK} is an isomorphism of functors on the category $\Gph$. 
This naturality is a powerful tool in applications \cite[\textsection 5]{AnDrummond-ColeKnudsen:SSGBG}.
On the other hand, the isomorphism obtained by combining Proposition \ref{prop:rho homotopy equivalence} and Observation \ref{obs:swiatkowski chains} is more geometric in nature.
Fortunately, we need not choose between these virtues.

\begin{proposition}\label{prop:swiatkowski comparison}
The diagram of isomorphisms 
\[\begin{tikzcd}H_*(B(\graf))\ar[dr,"\text{\rm Theorem }(\ref{thm:ADCK})\ "',"\simeq"]\ar[rr,"\text{\rm Proposition }(\ref{prop:rho homotopy equivalence})","\simeq"']&&H_*(UK(\graf))\ar[dl,"\ \text{\rm Observation }(\ref{obs:swiatkowski chains})","\simeq"']\\&H_*(S(\graf))\end{tikzcd}\]
commutes.
\end{proposition}

The proof of this result will occupy Section \ref{section:long ends and proof} below.
For now, we will use it to deduce the desired conclusion regarding edge stabilization.

For $\lambda\in A(\graf)$, write $e\lambda$ for the labeling that differs from $\lambda$ only in that $e\lambda(e)=\lambda(e)+1$ (in particular, $\lambda^{-1}(H)=(e\lambda)^{-1}(H)$).
There is a version of edge stabilization at the level of $UK(\graf)$ which sends $\lambda$ to $e\lambda$ and fixes all $t_h$ coordinates.
The induced $\mathbb{Z}[E]$-action on $C(UK(\graf))$ coincides, through the isomorphism of Observation \ref{obs:swiatkowski chains}, with the canonical action on $S(\graf)$.

\begin{proof}[Proof of Theorem \ref{thm:module iso}]
It suffices by Proposition \ref{prop:swiatkowski comparison} to show that the map \[\rho_*:H_*(B(\graf))\to H_*(\widetilde B(\graf))\to H_*(UK(\graf))\] is $\mathbb{Z}[E]$-linear.
By inspection, we have the commuting diagram 
\[\begin{tikzcd}
\widetilde B(\graf)\ar[d,dashed,"\sigma_e"']\ar[r]&UK(\graf)\ar[d]\\
\widetilde B(\graf)\ar[r]&UK(\graf),
\end{tikzcd}\] so the second map in this composite is $\mathbb{Z}[E]$-linear.
In order to show that the first map is also $\mathbb{Z}[E]$-linear, we note that the inverse, which is induced by the inclusion $\widetilde B(\graf)\subseteq B(\graf)$, is $\mathbb{Z}[E]$-linear.
\end{proof}

We also record the following useful conclusion.

\begin{corollary}\label{cor:chain level comparison}
The differential bigraded $\mathbb{Z}[E]$-modules $C^\sing(B(\graf))$ and $S(\graf)$ are quasi-isomorphic.
\end{corollary}
\begin{proof}
We have the zig-zag of $\mathbb{Z}[E]$-linear quasi-isomorphisms
\[
C^\sing(B(\graf))\xleftarrow{\sim} C^\sing(\widetilde B(\graf))\xrightarrow{\sim} C^\sing(UK(\graf))\xleftarrow{\sim}\bullet\xrightarrow{\sim} C(UK(\graf))\cong S(\graf).
\qedhere
\]
\end{proof}

Unlike the homology isomorphism, this quasi-isomorphism is not natural, since it relies on a choice of parametrization.

\section{Long ends and the proof of Proposition \ref{prop:swiatkowski comparison}}\label{section:long ends and proof}

In this section, we compare the two isomorphisms $H_*(B(\graf))\cong H_*(S(\graf))$.
The key observation is that, for a certain class of subdivision $\graf\to \graf'$, the natural map \[B_k^\Box(\graf')\subseteq B_k(\graf')\cong B_k(\graf)\to UK_k(\graf)\] from Abrams' model to \'{S}wi\k{a}tkowski's model is cellular.
With this observation in hand, Proposition \ref{prop:swiatkowski comparison} follows after checking that this collection of special subdivisions is large enough to support the argument of \cite{AnDrummond-ColeKnudsen:SSGBG}.

\subsection{Long ends}

In this section, we fix a parametrization of a graph $\graf$, identifying each edge with $(0,5)$. 
Given a subdivision $\graf\to \graf'$, we do \emph{not} independently parametrize $\graf'$. Rather, we use the homeomorphism underlying the subdivision to identify $\graf'$ with $\graf$. 
This allows us to specify the data of the subdivision $\graf'$ up to isomorphism by naming a finite set of points in $(0,5)$ for each edge of $\graf$.

\begin{definition}
We say that a subdivision $\graf\to \graf'$ has \emph{long ends} if, for every edge of $\graf$, the induced subdivision $\{0, t_1,\ldots, t_r,5\}$ of $[0,5]$ has $t_1=1$ and $t_r=4$.
\end{definition}

If $\graf'$ has long ends, we identify a half-edge $h$ of $\graf$ with the corresponding 1-cell of length $1$ in $\graf'$.
The following observation is the heart of our comparison of the two models in question. 

\begin{lemma}\label{lem:cellularity}
If $\graf\to \graf'$ has long ends, then the composite \[B_k^\Box(\graf')\subseteq B_k(\graf)\xrightarrow{\rho_k} UK_k(\graf)\] is cellular.
\end{lemma}
\begin{proof}
A cell of $B_k^\Box(\graf')$ is specified by a function $\mu$ from the set of cells of $\graf'$ to $\{0,1\}$ with the following properties:
\begin{enumerate}
\item $\sum_c\mu(c)=k$
\item if $c_i\cap c_j\neq \varnothing$, then $\mu(c_i)+\mu(c_j)\leq 1$.
\end{enumerate} The degree of the cell $\mu$ is the sum of $\mu$ over the 1-cells of $\graf'$.

By inspection, the composite in question maps the cell $\mu$ onto the cell $\lambda_\mu$ with \[\lambda_\mu(v)=\begin{cases}
v&\quad \mu(v)=1\\
h&\quad \mu(h)=1\\
\varnothing&\quad \text{otherwise.}
\end{cases}\] and $\lambda_\mu(e)$ defined by summing $\mu$ over all cells of the induced subdivision on $[1, 4]$. Since the degree of $\lambda_\mu$ is at most the degree of $\mu$, this claim implies the lemma.
\end{proof}
Write $\longends$ for the category of subdivisions of $\graf$ with long ends. Although this category is filtered, it is not convergent in the sense of \cite[Definition 2.6]{AnDrummond-ColeKnudsen:SSGBG}.
Nevertheless, we have the following.

\begin{lemma}\label{lem:long ends comparison}
For every $k\geq0$, the natural map \[\colim_{\longends}C(B_k^\Box(\graf'))\to \colim_{\subd}C(B_k^\Box(\graf'))\] is a quasi-isomorphism.
\end{lemma}
\begin{proof}
Let $\subd_k\subseteq \subd$ denote the subcategory of subdivisions of $\graf$ that are sufficiently subdivided for $k$, and set $\longends_k=\longends\cap \subd_k$.
Since $\longends_k$ is final in $\longends$ and $\subd_k$ in $\subd$, the vertical arrows in the commuting diagram 
\[\begin{tikzcd}
\colim_{\longends_k}C(B_k^\Box(\graf'))\ar[d]\ar[r]& \colim_{\subd_k}C(B_k^\Box(\graf'))\ar[d]\\
\colim_{\longends}C(B_k^\Box(\graf'))\ar[r]& \colim_{\subd}C(B_k^\Box(\graf'))
\end{tikzcd}\]
are isomorphisms, so it suffices to show that the top arrow is a quasi-isomorphism.
By sufficient subdivision and Theorem \ref{thm:abrams}, every arrow in $\subd_k$ induces a quasi-isomorphism on $C(B_k^\Box(-))$, so the claim follows from the observation that the inclusion of $\longends_k$ in $\subd_k$ induces a weak homotopy equivalence on nerves, since both nerves are contractible. Indeed, both $\longends$ and $\subd$ are filtered, hence contractible, and final functors induce weak equivalences on nerves.
\end{proof}

\begin{lemma}\label{lem:gap comparison}
Let $A=\pi^{-1}(A_0)$ be a gap.
If $A_0\cap \{0,1\}=\varnothing$, then the natural map \[\colim_{\graf'\in\longends\cap \subd_A}C(B^\Box(\graf'\setminus A))\to \colim_{\graf'\in\subd_A}C(B^\Box(\graf'\setminus A))\] is a quasi-isomorphism.
\end{lemma}
\begin{proof}
The category of subdivisions of $\graf'\setminus A$ that are sufficiently subdivided for a fixed $k$ is final in the category of subdivisions restricted from $\subd$.
Because $A_0\cap \{0,1\}=\varnothing$, the same is true for $\longends$, so the claim follows in the manner of Lemma \ref{lem:long ends comparison}.
\end{proof}

\begin{lemma}\label{lem:multi-basic compatibility}
Suppose that $\graf\to \graf'$ has long ends, and let $A$ be a gap in $\graf$ that is a union of cells of $\graf'$.
The following diagram of chain maps commutes:
\[\begin{tikzcd}
C(B^\Box(\graf'\setminus A))\ar[d]\ar[r]&C(B^\Box(\graf'))\ar[r,"C(\rho)"]&C(UK(\graf))\ar[d,equal,"\wr"]\\
\starreduced{\graf\setminus A}\ar[rr]&&S(\graf).
\end{tikzcd}\]
\end{lemma}
\begin{proof}
The claim is immediate from the explicit description of the value of $\rho$ on cells given in Lemma \ref{lem:cellularity} and the description of the lefthand vertical map given in \cite[Definitions 4.13 and 4.15]{AnDrummond-ColeKnudsen:SSGBG}.
\end{proof}

\subsection{Proof of Proposition \ref{prop:swiatkowski comparison}}
We wish to compare two isomorphisms identifying $H_*(B(\graf))$ with $H_*(S(\graf))$.
The first is induced on homology by the zig-zag of quasi-isomorphisms in the righthand portion of the diagram of Figure~\ref{figure:too-big-diagram}, in which each of the square subdiagrams commutes.
\begin{figure}[ht]
\[
\begin{tikzcd}[column sep=1.5em, row sep=1.2em]
\displaystyle\colim_{\graf'\in\longends} C^\sing(B^\Box(\graf'))
\ar[r]
&
\displaystyle\colim_{\graf'\in \subd} C^\sing(B^\Box(\graf'))
\ar[r,"\sim"]
&
C^\sing(B(\graf))
\\
\bullet
\ar[u,"(1)"]
\ar[d,"(2)"']
\ar[r]
& 
\bullet
\ar[u,"\wr"']
\ar[d,"\wr"]
\\
\displaystyle\colim_{\graf'\in\longends}C(B^\Box(\graf'))
\ar[r,"(3)"]
&
\displaystyle \colim_{\graf'\in\subd}C(B^\Box(\graf'))
\\
\displaystyle\hocolim_{A\in \gaps^{op}}\colim_{\graf'\in\longends\cap \subd_A}C(B^\Box(\graf'\setminus A))
\ar[r,"(4)"]
\ar[u]
&
\displaystyle\hocolim_{A\in \gaps^{op}}\colim_{\graf'\in\subd_A}C(B^\Box(\graf'\setminus A))
\ar[d,"\wr"]
\ar[u,"\wr"']
\\
&
\displaystyle\hocolim_{A\in \gaps^{op}}\starreduced{\graf\setminus A}
\ar[r,"\sim"]
&
S(\graf)
\end{tikzcd}
\]
\caption{A commutative diagram of quasi-isomorphisms for the proof of Proposition \ref{prop:swiatkowski comparison}}
\label{figure:too-big-diagram}
\end{figure} 
The first step in the proof is to verify that \emph{all} of the maps in the diagram are quasi-isomorphisms, so that we may replace this zig-zag with the outer zig-zag in the diagram. 
\begin{enumerate}
\item[(1)--(2)] These maps are induced by natural quasi-isomorphisms after taking the colimit over $\longends$.
Since $\longends$ is filtered, the claim follows.
\item[(3)] This quasi-isomorphism is supplied by Lemma \ref{lem:long ends comparison}.
\item[(4)] This quasi-isomorphism follows from Lemma \ref{lem:gap comparison} and the observation that the poset of gaps satisfying the hypotheses of that lemma is homotopy initial in $\gaps$ (hence homotopy final in $\gaps^{op}$).
\end{enumerate} The remaining arrows are quasi-isomorphisms by two-out-of-three.

The second isomorphism $H_*(B(\graf))\to H_*(S(\graf))$ is induced by the map $\rho:B(\graf)\to UK(\graf)$ of topological spaces, together with the identification of $S(\graf)$ with cellular chains on $UK(\graf)$.
In order to compare this isomorphism with the previous, we note that Lemma \ref{lem:cellularity} supplies the dashed fillers in the commuting diagram 
\[\begin{tikzcd}[column sep=1.5em, row sep=1.2em]
C^\sing(B^\Box(\graf'))\ar[r]&C^\sing(B(\graf))\ar[r,"\sim"]&C^\sing(UK(\graf))\\
\bullet\ar[rr,dashed]\ar[u,"\wr"]\ar[d,"\wr"']&&\bullet\ar[u,"\wr"']\ar[d,"\wr"]\\
C(B^\Box(\graf'))\ar[rr,dashed]&&C(UK(\graf))
\end{tikzcd}\] whenever $\graf\to \graf'$ has long ends.
Passing to the colimit over $\longends$, we obtain the diagram
\[\begin{tikzcd}[column sep=1.5em, row sep=1.2em]
\displaystyle\colim_{\graf'\in\longends} C^\sing(B^\Box(\graf'))\ar[r,"\sim"]&C^\sing(B(\graf))\ar[r,"\sim"]&C^\sing(UK(\graf))\\
\bullet\ar[u,"\wr"]\ar[d,"\wr"']\ar[rr]&&\bullet\ar[u,"\wr"']\ar[d,"\wr"]\\
\displaystyle\colim_{\graf'\in\longends} C(B^\Box(\graf'))\ar[rr]&&C(UK(\graf))\ar[dd,equal,"\wr"]\\
\displaystyle\hocolim_{A\in \gaps^{op}}\colim_{\graf'\in\longends\cap \subd_A}C(B^\Box(\graf'\setminus A))\ar[d,"\wr"']\ar[u,"\wr"]\\
\displaystyle\hocolim_{A\in \gaps^{op}}\starreduced{\graf\setminus A}\ar[rr,"\sim"]&&S(\graf).
\end{tikzcd}\]
The upper portion of the diagram commutes by what has already been said, and the bottom portion of the diagram commutes by Lemma \ref{lem:multi-basic compatibility} and the universal properties of the colimit and the homotopy colimit.
We established in the first half of the proof that the first of the isomorphisms in question is induced on homology by the counterclockwise zig-zag from $C^\sing(B(\graf))$ to $S(\graf)$.
Since the clockwise zig-zag induces the second of the isomorphisms, the proof is complete.

\bibliographystyle{gtart}
\bibliography{references}
\end{document}

%% file: deformation_input.tex
\begingroup%
  \makeatletter%
  \providecommand\color[2][]{%
    \errmessage{(Inkscape) Color is used for the text in Inkscape, but the package 'color.sty' is not loaded}%
    \renewcommand\color[2][]{}%
  }%
  \providecommand\transparent[1]{%
    \errmessage{(Inkscape) Transparency is used (non-zero) for the text in Inkscape, but the package 'transparent.sty' is not loaded}%
    \renewcommand\transparent[1]{}%
  }%
  \providecommand\rotatebox[2]{#2}%
  \ifx\svgwidth\undefined%
    \setlength{\unitlength}{294.96638489bp}%
    \ifx\svgscale\undefined%
      \relax%
    \else%
      \setlength{\unitlength}{\unitlength * \real{\svgscale}}%
    \fi%
  \else%
    \setlength{\unitlength}{\svgwidth}%
  \fi%
  \global\let\svgwidth\undefined%
  \global\let\svgscale\undefined%
  \makeatother%
  \begin{picture}(1,1.01291436)%
    \put(0,0){\includegraphics[width=\unitlength,page=1]{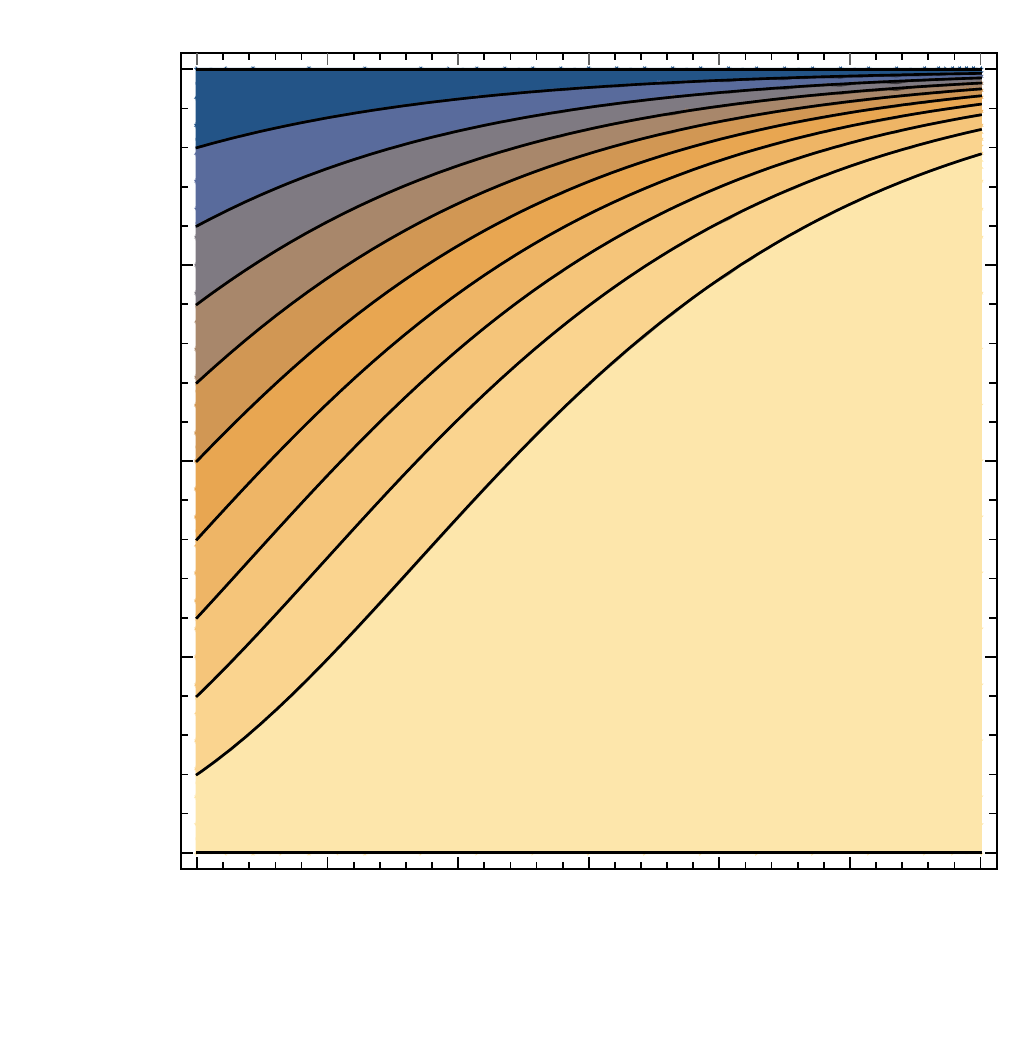}}%
    \put(0.11880681,0.17021538){\color[rgb]{0,0,0}\makebox(0,0)[lb]{\smash{$0$}}}%
    \put(0.12006638,0.55272719){\color[rgb]{0,0,0}\makebox(0,0)[lb]{\smash{$1$}}}%
    \put(0.1150281,0.93522577){\color[rgb]{0,0,0}\makebox(0,0)[lb]{\smash{$2$}}}%
    \put(0.1599016,0.08905058){\color[rgb]{0,0,0}\makebox(0,0)[lb]{\smash{$0$}}}%
    \put(0.41490947,0.08905058){\color[rgb]{0,0,0}\makebox(0,0)[lb]{\smash{$1$}}}%
    \put(0.6699041,0.08905058){\color[rgb]{0,0,0}\makebox(0,0)[lb]{\smash{$2$}}}%
    \put(0.92491197,0.08905058){\color[rgb]{0,0,0}\makebox(0,0)[lb]{\smash{$3$}}}%
    \put(0.01014748,0.55272719){\color[rgb]{0,0,0}\makebox(0,0)[lb]{\smash{$s$}}}%
    \put(0.54651872,0.0186892){\color[rgb]{0,0,0}\makebox(0,0)[lb]{\smash{$t$}}}%
  \end{picture}%
\endgroup%